\documentclass[11pt]{amsart}

\usepackage{epsfig,amsmath,amsfonts,latexsym}
\usepackage[abs]{overpic}
\usepackage{color}
 \usepackage{palatino}
 \usepackage{marginnote}
 \usepackage{hyperref}
\usepackage[all,cmtip]{xy} 
\usepackage{tikz-cd}
\usepackage[mathscr]{euscript}
\usepackage{graphicx}

\newtheorem{claim}{Claim}
\newtheorem{question}{Question}
\newtheorem{conjecture}{Conjecture}
\newtheorem{theorem}{Theorem}
\newtheorem{proposition}[theorem]{Proposition}
\newtheorem{lemma}[theorem]{Lemma}

\theoremstyle{definition}
\newtheorem{definition}[theorem]{Definition}

\theoremstyle{remark}
\newtheorem{remark}[theorem]{Remark}

\newcommand{\Diff}{\mathrm{Diff}}

\newcommand{\Pin}{\mathrm{Pin}(2)}

\newcommand{\FBF}{\mathop{\mathrm{FBF}}\nolimits}

\newcommand{\id}{\mathrm{id}}

\newcommand{\s}{\mathfrak{s}}

\def\Spinc{\text{spin}^c}

\def\ker{\operatorname{Ker}}

\numberwithin{theorem}{section}
\theoremstyle{plain}
\title{Family Bauer--Furuta invariant, Exotic Surfaces and Smale conjecture}
\author{Jianfeng Lin, Anubhav Mukherjee}

\address{Yau Mathematical Sciences Center\\ Tsinghua University\\ Beijing\\ China.}

\email{linjian5477@gmail.com}

\address{School of Mathematics \\ Georgia Institute
of Technology \\  Atlanta  \\ Georgia\\ USA.}
\email{anubhavmaths@gatech.edu}

\begin{document}

\maketitle
\begin{abstract}
We establish the existence of a pair of exotic surfaces in a punctured $K3$ which remains exotic after one external stabilization and have diffeomorphic complements. A key ingredient in the proof is a vanishing theorem of the family Bauer--Furuta invariant for diffeomorphisms on a large family of spin 4-manifolds, which is proved using the tom Dieck splitting theorem in equivariant stable homotopy theory. In particular, we prove that the $S^{1}$-equivariant family Bauer--Furuta invariant of any orientation-preserving diffeomorphism on $S^{4}$ is trivial and that the $\Pin$-equivariant family Bauer--Furuta invariant for a diffeomorphism on $S^{2}\times S^{2}$ is trivial if the diffeomorphism acts trivially on the homology. Therefore, these invariants do not detect exotic self-diffeomorphisms on $S^{4}$ or $S^{2}\times S^{2}$. Furthermore, our theorem also applies to certain exotic loops of diffeomorphisms on $S^{4}$ (as recently discovered by Watanabe) and show that these loops have trivial family Bauer--Furuta invariants. En route, we observe a curious element in the $\Pin$-equivariant stable homotopy group of spheres which could potentially be used to detect an exotic diffeomorphism on $S^{4}$.     
\end{abstract}

\section{Introduction}
The Bauer--Furuta invariant of smooth 4-manifolds is a stable homotopy refinement of the classical Seiberg--Witten invariant \cite{BauerFuruta1}. It contains subtle but important information about the underlying 4-manifolds which is invisible from other invariants. In particular, the Bauer--Furuta invariant can detect new exotic smooth structures (for example, on $K3\#K3$) and its family version \cite{Szymik2010, Xu2004} is a powerful tool in detecting exotic diffeomorphisms \cite{KM2020,lin}. For this reason, this invariant has been actively studied in recent years \cite{baraglia2019note,baraglia2019bauer,bryan1997seiberg,iida2019bauer,konno2021involutions,KM2020,Schmidt2003spin,Szymik2008,Szymik2012}. While the classical Seiberg--Witten  invariant is an integer, the Bauer--Furuta invariant is an element in the (equivariant) stable homotopy group of spheres. This allows one to apply powerful techniques in equivariant topology (e.g. equivariant K-theory and equivariant homology) to extract information on the Bauer--Furuta invairiant and to study 4-manifolds. This  idea has resulted in many remarkable theorems including Furuta's $\frac{10}{8}$-theorem \cite{furuta} (see also \cite{furuta2007homotopy,hopkinslinshixu}) and Manolescu's disproof of the triangulation conjecture \cite{CM16} (see also \cite{lin2018morse}). In this paper, we follow this line of thought and use the equivariant stable homotopy theory to to prove a general vanishing result for the family Bauer--Furuta invariant of diffeomorphism on a large family of spin 4-manifolds (Theorem \ref{thm: generalized vanishing theorem}). This vanishing theorem has several applications, which we discuss first.

\subsection{Exotic surfaces} The first application concerns exotic surfaces in 4-manifolds. Recall that a pair of compact surfaces $\Sigma, \Sigma'$ in a 4-manifold $X$ with $\partial\Sigma=\partial\Sigma'\subset \partial X$ are called \textit{exotic surfaces} if there exists a topological isotopy of $X$ relative to the boundary that takes $\Sigma$ to $\Sigma'$ but there is no such smooth isotopy. Understanding exotic surfaces in 4-manifolds 
is of great interest since they are closely related to exotic smooth structures on 4-manifolds. In recent years, various methods have been used to detect examples of exotic surfaces (e.g., the Donaldson's invariants \cite{d-inv}, the Seiberg--Witten invariants \cite{fs97}, Heegaard-Floer homology \cite{JMZ}, symplectic \& contact geometry \cite{hayden} and Khovanov homology \cite{hayden2021khovanov}). In this paper, we add the family Bauer--Furuta invariant to this list. As we will see, this method reveals a new phenomenon 
: some exotic surfaces can survive a stabilization. 

To further explain this, we mention an important principle in 4-dimensional topology, discovered by Wall in the 1960s \cite{wall}. This principle says that exotic phenomena are all eliminated by sufficiently many stabilizations, i.e., taking connected sum with $S^{2}\times S^{2}$'s. Since then, it has been a fundamental problem to search for exotic phenomena that survives one stabilization. In the case of exotic surfaces, work of Perron \cite{PERRON1} and Quinn \cite{Quinn} showed that exotic surfaces always become smoothly isotopic after taking sufficiently many stabilizations (
away from the surfaces). 
\textit{It is a long standing open-question, whether a single stabilization is enough for this purpose?} (See \cite{Akbulut2014,AKHM} for studies of this problem.) Using the 4-dimensional light bulb theorem (proved by Gabai \cite{gabai} and  Scheniderman--Teichner \cite{Schneiderman2019HomotopyVI}),  Auckly--Kim--Melvin--Ruberman--Schwartz \cite{AKMRS} proves a surprising result in the positive direction: They showed that for exotic pairs of connected surfaces in non-characteristic homology classes with simply-connected complements, a single stabilization is enough. In this paper, we prove a result in the negative direction by producing the first pair of orientable exotic surfaces 
that remain non-isotopic after a single stabilization. 

\begin{theorem}\label{surface}
 There exists a link $L$ on the boundary $S^3$ of punctured $K3$ that bounds a pair of slice disks $D_L$ and $D'_L$ which are topologically isotopic relative to the boundary but not smoothly so. Moreover, they remain exotic after a single stabilization, i.e. connected sum with $S^2\times S^2$, and their complements are diffeomorphic. \end{theorem}


\begin{remark}
As mentioned in Theorem \ref{surface}, another interesting feature of our example is that the complements of $D_{L}$ and $D_{L}'$ are diffeomorphic. Examples with this feature was produced by Schwartz \cite{Schwartz19} (for closed manifolds, using non-triviallity of an h-cobordism) and Akbulut \cite{akbulut2021corks} (for embedded disks in the 4-ball, using  corks). Note that recent work of Juh\'asz--Miller--Zemke \cite{JMZ} and Hayden \cite{hayden} give some fascinating examples of knots in $S^3$ that bounds exotic surfaces in 4-ball. In those examples it is not clear if the complements of the surfaces are diffeomorphic or not.  
\end{remark}

We sketch the proof of Theorem \ref{surface} as follows (see Section~\ref{proof of exoticness} for details). Let $D_{L}=\bigsqcup_{22} D^{2}$ be the disjoint union of the cores of 2-handles in the punctured $K3$ and let $D'_{L}$ be the image of $D_{L}$ under the Dehn twist on the punctured $K3$ along its boundary. By a result in  \cite{krannich2021torelli} and \cite{OP},  this Dehn twist is topologically isotopic to the identity map relative to the boundary. Hence $D_{L}$ and $D'_{L}$ are topologically isotopic relative to the boundary. However, $D_{L}$ and $D'_{L}$ are not smoothly isotopic relative to the boundary even after a stabilization. Suppose they were. Then we show that there is a smooth isotopy between $\delta\# \id_{S^{2}\times S^{2}}$ and $\id_{K3\# K3}\# f$, where $\delta$ denote the Dehn twist on $K3\# K3$ and $f$ is a diffeomorphism $S^{2}\times S^{2}$. 
By the connected sum formula of the Bauer--Furuta invariant and our vanishing theorem (Theorem \ref{thm: generalized vanishing theorem}) in the case of $S^{2}\times S^{2}$, the $\Pin$-equivariant family Bauer--Furuta invariant of $\id_{K3\# K3}\# f$ is trivial. This contradicts with a result of the first author \cite{lin}, which states that the corresponding invariant of $\delta\# \id_{S^{2}\times S^{2}}$ is nontrivial. 

 It is worth noting that the argument in the current paper is essentially different from the argument in \cite{lin} by the first author. In \cite{lin}, there is a specific diffeomorphism $\delta\# \id_{S^{2}\times S^{2}}$. By partially computing it's family Bauer--Furuta invariant, one can show that the invariant is non vanishing. However, in the current paper, we just have a map of the form $\id_{K3\# K3}\# f$ with $f$ being a hypothetical diffeomorphism on $S^{2}\times S^{2}$. Since we don't have much control over $f$, we have to prove a general vanishing theorem for the family Bauer--Furuta invariant  that works for arbitrary $f$ that acts  trivially on homology. This is where extra arguments in equivariant stable homotopy theory have to get involved.

We also emphasize that in the proof of Theorem \ref{surface}, the $\Pin$-equvariance is essentially used multiple times for different purposes. First, the $\Pin$-symmetry is included in the definition of the equivariant Bauer--Furuta invariant, so that this invariant can be nonzero even after a stabilization. This special feature, which is different from many other gauge theoretic invariants (e.g., the Seiberg--Witten invariant, Donaldson's polynomial invariant, the $S^{1}$-equivariant and the non-equivariant Bauer--Furuta invariant), makes the $\Pin$-equivariant Bauer--Furuta invariant particularly suitable for detecting exotic phenomena under stabilizations. Second, because of this $\Pin$-symmetry, the map between spheres which underlies the Bauer--Furuta invariant must be a $\Pin$-equivariant map. This actually imposes a very strong constraint on the possible values of the Bauer--Furuta invariant. This constrain is crucial in the proof of our vanishing theorem.  

\begin{remark}
From the construction of $D_{L}$ and $D'_{L}$, we see that $L$ can be taken to be the attaching link in any Kirby diagram of $K3$ (without $1$-handles and $3$-handles). Different choices of $L$ will give different exotic surfaces. The  following proposition gives an upper bound of the stabilizations needed to make these surfaces smoothly isotopic. It is proved by combining the argument in \cite{AKMRS} with the recent work of Schwartz \cite{schwartz2021}, which generalizes the light bulb theorem to embedded disks.  
 \begin{proposition}\label{upper bound on stabilizations}
(1) Let $L_{0}\subset L$ be a sublink and let $D_{L_{0}}\subset D_{L}$ and $D'_{L_{0}}\subset  D'_{L}$ be the subsurfaces that contain all the components with boundary in $L_{0}$. Suppose $L_{0}$ is an unlink. Then $D_{L_{0}}$ is smoothly isotopic to $D'_{L_{0}}$ relative to the boundary (without any stabilizations). 

(2) Suppose $L$ contains an $m$-component unlink ($m\geq 0$). Then after $22-m$ stabilizations, $D_{L}$ and $D'_{L}$ are smoothly isotopic relative to the boundary. In particular, for any $L$, $D_{L}$ and $D'_{L}$ are smoothly isotopic after $22$ stabilizations.
\end{proposition}

In Section \ref{proof of exoticness}, we give an example of $L$ with all components being the unknot. By Proposition \ref{upper bound on stabilizations}, each component of $D_{L}$ is smoothly isotopic to the  corresponding component  of $D'_{L}$. Therefore,  in this example, the  exotica that persists a stabilization is the knotting between different components.  Proposition \ref{upper bound on stabilizations} also motivates the following question,
 
 \begin{question}
 Given $L$, what is the minimal number of stablizations that can make $D_L$ and $D'_L$ smoothly isotopic?
 \end{question}

\end{remark}

\subsection{Smale Conjecture} Our second motivation comes from a famous conjecture by Smale. In dimension 4, this conjecture states that the natural inclusion $O(5)\hookrightarrow \operatorname{Diff}(S^4) $ is a homotopy equivalence [Kirby's Problem List~\cite{kirby}, Problem 4.34, 4.126]. Recently,  Watanabe \cite{watanabe} disproved Smale Conjecture by showing that the induced maps on higher homotopy groups are not surjective. Here we state a special case of his 
result which is most relevant to our discussion:
\begin{theorem}[Watanabe, \cite{watanabe}] There is a non-torsion element $\gamma_{1}\in \pi_{1}(\operatorname{Diff}(S^{4}))$ such that for any nonzero integer  $d$, one has 
$$
d\gamma_{1} \notin \operatorname{image}(\pi_{1}(O(5))\rightarrow \pi_{1}(\operatorname{Diff}(S^{4}))).
$$ Furthermore, $\gamma_{1}$ is an exotic loop of diffeomorphisms. Namely, $\gamma_{1}$ is null-homotopic in the homeomorphism group of $S^{4}$.
\end{theorem}
We note that the following important case is still open:

\begin{conjecture}
[Smale]\label{conjecture: smale}
Any orientation preserving self-diffeomorphism of $S^4$ is smoothly isotopic to Identity.
\end{conjecture}

Watanabe's proof is not gauge theoretic. Instead, he used  Kontsevich's characteristic classes to detect $\gamma_{1}$. Note that gauge theory has also been very successful in detecting exotic diffeomorphisms on 4-manifolds \cite{Akbulut2014,AKHM,BK20, KM2020, lin, rubermanexotic}. In an upcoming paper by Auckly--Ruberman \cite{Aucklyexoticfamily}, they succeeded in using gauge theory to detect exotic loops (and elements in higher homotopy groups) in the space of diffeomorphisms of a 4-manifold. Given these developments, one may ask the following two general questions:
\begin{question}\label{ques: exotic loop detection}
Can $\gamma_{1}$ and its multiples be detected by gauge theoretic invariants?   
\end{question}
\begin{question}\label{ques: exotic diff detection}
Can we use gauge theoretic invariants to attack Conjecture~\ref{conjecture: smale}? (One can see Akbulut's recent work~\cite{akbulut2021smale} towards this conjecture.)
\end{question}

Note that a negative answer to Question \ref{ques: exotic loop detection} would also be very interesting because it implies that even on $S^{4}$, there exist exotic phenomena that are invisible from gauge theory. 

Since the family Bauer--Furuta invariant is a homotopy refinment of the family Seiberg--Witten invariant \cite{baraglia2019bauer,BauerFuruta1}, it is natural to specialize these questions to the the family Bauer--Furuta invariants. Our results (Theorem \ref{theorem main: vanishing result for S4} and Theorem \ref{theorem main: vanishing result for loops}) provide partial negative answers to these questions.
 
\begin{theorem}\label{theorem main: vanishing result for S4}
Let $f$ be an orientation-preserving diffeomorphism of $S^4$.
Let $\widetilde{f}$ be a lift of $f$ to an automorphism of the spin structure of $S^4$.
Then we have 
\[
\FBF^{S^{1}}(S^{4},f,\widetilde{f})=0\quad\text{and}\quad \FBF^{\{e\}}(S^{4},f,\widetilde{f})=0.
\]
\end{theorem}

 \begin{theorem}\label{theorem main: vanishing result for loops} Let $\gamma$ be an exotic loop in $\operatorname{Diff}(S^{4})$ then one has  
\[
\FBF^{S^{1}}(S^{4},\gamma)=0\quad\text{and}\quad \FBF^{\{e\}}(S^{4},\gamma)=0.
\]
If we further  assume that $\gamma$ is  divisible by $2$ in $\pi_{1}(\operatorname{Diff}(S^{4}))$ (e.g., $\gamma=2k \gamma_{1}$), then \[
\FBF^{\Pin}(S^{4},\gamma)=0.\]
\end{theorem}

Theorem \ref{theorem main: vanishing result for S4} and Theorem \ref{theorem main: vanishing result for loops} together imply that  the $S^1$-equivariant version and the non-equivariant version of the family Bauer--Furuta invariant do not detect exotic (loops of) diffeomorphisms on $S^{4}$. In particular, Theorem \ref{theorem main: vanishing result for S4} rules out the possibility that the non-equivariant Bauer--Furuta invariant of a diffeomorphism (which is an element in $\pi^{s}_{1}(S^{0})$) equals the Hopf map. And Theorem \ref{theorem main: vanishing result for loops} rules out the possibility that the non-equivariant Bauer--Furuta invariant of an exotic loop of diffeomorphisms (which is an element in $\pi^{s}_{2}(S^{0})$) equals the square of the Hopf map.

Note that Theorem \ref{theorem main: vanishing result for S4} has a generalization (Theorem \ref{thm: generalized vanishing theorem}) that applies to a large family of spin 4-manifolds. However, we do not have a vanishing result for the $\Pin$-equivariant family Bauer--Furuta invariants that applies to all diffeomorphisms on $S^{4}$. Actually, we discovered a curious 2-torsion element in the $\Pin$-equivariant stable homotopy group of spheres. This element satisfies a basic requirement, called \textit{``BF-type''}, of being realizable as $\FBF^{\Pin}(S^{4},f,\widetilde{f})$ for some exotic diffeomorphism $f$. Roughly speaking, this condition means that the equivariant homotopy class is trivial 
when restricted to the $S^{1}$-fixed points. We refer to Definition \ref{definition: elements of BF-type} for a precise definition. 
\begin{theorem}\label{theorem main: curious Pin(2)-element}
There is an unique nonzero element $\alpha_{0}\in \pi^{\Pin}_{1,\mathcal{U}}(S^{0})$ that is of BF-type. 
\end{theorem}
Theorem \ref{theorem main: curious Pin(2)-element} suggests that the $\Pin$-equivariant family Bauer--Furuta invariant could potentially be useful to detect a counter example of Conjecture \ref{conjecture: smale}. It would be extremely interesting to calculate the $\Pin$-equivariant family Bauer--Furuta invariant of existing potential counter examples \cite{budney2019knotted,gay2021}. Furthermore, Theorem \ref{theorem main: vanishing result for S4} and Theorem \ref{theorem main: curious Pin(2)-element} together provide strong evidence indicating that the $j$-symmetry is an indispensable ingredient if Seiberg--Witten theory can be used to attack Conjecture \ref{conjecture: smale}. For this reason, it is natural to try to detect this element $\alpha_{0}$ by studying the induced map on the $G$-equivariant Borel (co)homology for various subgroups $G$ of $\Pin$, because they correspond to various versions of monopole Floer homology. However, we will give an explicit space level description of $\alpha_{0}$ in Section \ref{section: properties of FBF} and prove that 
 this method does not detect $\alpha_{0}$. 
\subsection{The vanishing theorem} Now we state the vanishing theorem for the family Bauer--Furuta invariant of a diffeomorphism. As we mentioned, this theorem generalizes Theorem \ref{theorem main: vanishing result for S4} and it plays an essential role in the proof of Theorem \ref{surface}.
\begin{theorem}\label{thm: generalized vanishing theorem}
Let $X$ be an oriented spin 4-manifold with $b_{1}(X)=0$. We fix a homology orientation and a spin structure on $X$. Let $f:X\rightarrow X$ be a diffeomorphism which preserves the spin structure and acts trivially on $H^{*}(X;\mathbb{R})$. Then for any lift  $\widetilde{f}$ of $f$ to an automorphism on the spin structure, we have the following results:
\begin{enumerate}
    \item Suppose $2\chi(X)+3\sigma(X)=4$. Then the invariants $$\FBF^{S^{1}}(X,f,\widetilde{f})\quad \text{and }\quad \FBF^{\{e\}}(X,f,\widetilde{f})$$ are both vanishing.
    \item Suppose $2\chi(X)+3\sigma(X)>4$. Then the invariants   $$\FBF^{\Pin}(X,f,\widetilde{f}),\quad \FBF^{S^{1}}(X,f,\widetilde{f})\quad \text{and}\quad \FBF^{\{e\}}(X,f,\widetilde{f})$$ are all vanishing.
\end{enumerate}
\end{theorem}
We note that for any spin 4-manifold, we have $ 4\mid 2\chi(X)+3\sigma(X)$. We also note that the quantity $2\chi(X)+3\sigma(X)$, which equals $c_{1}(TX)^{2}[X]$ when $X$ is equipped with an almost complex structure, plays an important role in the geography problem of irreducible and symplectic 4-manifolds (see, for example, \cite{FS2009}). Actually, all known examples of simply-connected, irreducible 4-manifolds satisfies the inequality
\begin{equation}\label{positive chern number}
 2\chi(X)+3\sigma(X)\geq 0.   
\end{equation}
 Moreover, by a result of Taubes \cite{Taubes2000}, the inequality (\ref{positive chern number}) holds for any irreducible symplectic 4-manifolds. In particular, this implies that Theorem \ref{thm: generalized vanishing theorem} applies to any simply-connected, spin, symplectic 4-manifold that is not homeomorphic to an elliptic surface $E(2k)$. And it also applies to many reducible 4-manifolds (e.g. $K3\#(S^{2}\times S^{2})$). We refers to Remark \ref{remark: nonvanishing FBF} for examples where $2\chi(X)+3\sigma(X)=-4, 0$ or $4$ but $\FBF^{\Pin}(X,f,\widetilde{f})\neq 0$.  
 
 Theorem \ref{thm: generalized vanishing theorem} is proved using tools from equivariant stable homotopy theory (in particular, the tom Dieck splitting theorem and the Adams' isomorphism). The geometric input from gauge theory is the following observation: When restricted to the $S^{1}$-fixed points, the family Bauer--Furuta invariant can be defined via the linear operator $d^{+}\oplus d^{*}$, which only records the induced map on $H^{2}(X;\mathbb{R})$ and hence can not distinguish $f$ with the identity map on $X$. As a result, the family Bauer--Furuta invariant becomes trivial once we restrict it to the $S^{1}$-fixed points. In a certain range of dimensions, this is enough to show that the family Bauer--Furuta invariant itself is trivial even before passing to the $S^{1}$-fixed points. So we can prove Theorem \ref{thm: generalized vanishing theorem}.

 As a final remark, we mention that Theorem \ref{thm: generalized vanishing theorem} doesn't say anything about those $\Spinc$ structures not coming from a spin structure. Actually, the family Bauer--Furuta invariants and the family Seiberg--Witten invariants for these $\Spinc$ structures do detect exotic diffeomorphisms (see \cite{BK20,rubermanexotic}). Note that when defining the family Bauer--Furuta invariant in the absence of a spin structure, there are some extra complications in choosing the framing of the moduli space (see \cite[Section 3]{KM2020}).

\textbf{Organization of the paper:} In Section~\ref{section: definition of FBF}, we recall the definition of the $\Pin$-equivariant family Bauer--Furuta invariant. Then in Section~\ref{proof of exoticness}, we prove Theorem \ref{surface} (existence of exotic surfaces under a stabilization) by assuming that Theorem \ref{thm: generalized vanishing theorem} (the vanishing theorem for the family Bauer--Furuta invariant) holds for $S^{2}\times S^{2}$. We also prove Proposition \ref{upper bound on stabilizations} in this section. The rest of the paper is devoted to the proof of the vanishing results Theorem \ref{theorem main: vanishing result for loops},  Theorem \ref{theorem main: curious Pin(2)-element} and Theorem \ref{thm: generalized vanishing theorem}. In Section~\ref{section homotopy theory}, we first review some standard constructions and results in equivariant stable homotopy theory. Then we prove several algebraic vanishing results for the $\Pin$-equivariant stable homotopy groups (Proposition \ref{prop: algebraic vanishing result1} $\sim$ Proposition \ref{prop: algebraic vanishing result4}). These vanishing results will be used in Section~\ref{section: properties of FBF} to finish the proof of Theorem \ref{theorem main: vanishing result for loops}, Theorem \ref{theorem main: curious Pin(2)-element} and Theorem \ref{thm: generalized vanishing theorem}. Section~\ref{section homotopy theory} and \ref{section: properties of FBF} mainly consist of algebraic topology.\\ 



\textbf{Acknowledgments:} The authors are extremely thankful to John Etnyre and Danny Ruberman for having conversations \& email exchanges and providing some key insights related to this work. The first  author wishes to thank Mike Hill and XiaoLin Danny Shi for pointing out the references \cite{blumberg2021equivariant,Lewis2000}. The second author is especially grateful to Nobuo Iida, Hokuto Konno and  Masaki Taniguchi for their priceless friendship and being patience while teaching him Bauer--Furuta theory and having many discussions related to this work. We are  thankful to Kyle Hayden, Hyunki Min, Lisa Piccirillo and Ian Zemke for showing great interest in this project. We are also thankful to Selman Akbulut, Sander Kupers and Markus Szymik for giving helpful comments on an earlier version of the paper. The second author is partially supported by the NSF grant DMS-1906414.

\section{The Family Bauer--Furuta Invariant.}\label{section: definition of FBF}
 Our proof of Theorem \ref{surface} heavily uses the family Bauer--Furuta invariant, which was defined by Bauer--Furuta \cite{BauerFuruta1} (for a single 4-manifold), Szymik \cite{Szymik2010} and Xu \cite{Xu2004} (for families). So for completeness, we first recall its definition following the expositions in \cite{KM2020, lin}.  
We start by introducing some notations. Consider the 1-dimensional Lie group 
$$
\Pin=\{e^{i\theta}\}\cup \{j\cdot e^{i\theta}\}\subset \mathbb{H}.
$$
We are mainly interested in the following three representations of $\Pin$:
\begin{itemize}
    \item $\mathbb{H}$: the 4-dimensional representation, where $\Pin$ acts as the left multiplication. 
    \item $\widetilde{\mathbb{R}}$: the 1-dimensional representation, where the unique component $S^{1}\subset \Pin$ acts trivially and the other component acts as multiplication by $-1$. Note that this is also the adjoint representation of $\Pin$ on its Lie algebra.
    \item $\mathbb{R}$: the 1-dimensional trivial representation.
\end{itemize}
For a finite dimensional $\Pin$-representation $V$, we use $S^{V}$ to denote the one point compactification of $V$ and use  $S(V)$ to denote the unit sphere in $V$. The space $S^{V}$ is called the ``representation sphere'' and it admits a $\Pin$-action that fixes the point $\infty$, which we choose as the base point. 
\subsection{The invariant of a single diffeomorphism}
Now we start the construction of the family Bauer--Furuta invariant of a self-diffeomorphism. Let $X$ be an oriented 4-manifold with $b_{1}(X)=0$. We equip $X$ with a cohomological orientation (i.e. an orientation of $H^{2}_{+}(X;\mathbb{R})$) and a spin structure. We use $P$ to denote the corresponding $\mathrm{Spin}(4)$-bundle. Let $f$ be a self-diffeomorphism on $X$ such that $f^{*}$ acts trivially on $H^{*}(X;\mathbb{R})$. Then we form the mapping torus 
$$
Tf=([0,1]\times X)/(0,x)\sim (1,f(x))
$$
and treat it as a fiber bundle over $S^{1}$ with fiber $X$. Next, let $\widetilde{f}:P\rightarrow P$ be a bundle automorphism of $P$ which lifts $f$. (There are two choices of $\widetilde{f}$ and we pick one.) Then we form the $\mathrm{Spin}(4)$-bundle $$([0,1]\times P)/(0,x)\sim (1,\widetilde{f}(x)),$$ which gives a family spin structure on $Tf$. This family spin structure allows us to define the family Seiberg--Witten equations on $Tf$. By doing the \textit{finite dimensional approximation} on these equations, we obtain a $\Pin$-equivariant map 
\begin{equation}\label{eq: sw}
\widetilde{sw}: U\rightarrow V
\end{equation}
, where $V$ is finite dimensional representation space of $\Pin$ and $U$ is a fiber bundle over $S^{1}$ whose fibers are finite dimensional representations spaces of $\Pin$. Moreover, because of the compactness property of the Seiberg--Witten moduli space, by choosing $U,V$ large enough, one can ensure that $\widetilde{sw}$ satisfies the following nice property: There exists a large $R$ and a small $\epsilon$, such that 
$$
\widetilde{sw}(S_{R}(U))\cap B_{\epsilon}(V)=\emptyset.
$$
Here $S_{R}(U)\subset U$ denotes the sphere bundle with radius $R$ and $B_{\epsilon}(V)\subset V$ denotes the  disk with radius $\epsilon$. As a consequence, we get an induced map 
$$
\widetilde{sw}^{+}:U^{+}\cong B_{R}(U)/S_{R}(U)\rightarrow V/(V\setminus \operatorname{int}(B_{\epsilon}(V)))\cong V^{+}
$$
between one-point compactifications.
Furthermore, as explained in \cite{KM2020,lin}, by using the homology orientation on $X$ and the $\Pin$-actions on $U,V$, one can choose  canonical (up to homotopy) identifications 
\begin{equation}\label{eq: trivialization}
U\cong S^{1}\times (\widetilde{\mathbb{R}}^{n'}\oplus \mathbb{H}^{m'-\frac{\sigma(X)}{16}}),\quad V\cong \widetilde{\mathbb{R}}^{n'+b^{+}_{2}(X)}\oplus \mathbb{H}^{m'}
\end{equation}
for some $m',n'\gg 0$. Using these trivilizations, we can rewrite  $\widetilde{sw}^{+}$ as a map
\begin{equation}\label{eq: SW+}
\widetilde{sw}^{+}: S^{1}_{+}\wedge   S^{(m'-\frac{\sigma(X)}{16})\mathbb{H}+  n'\widetilde{\mathbb{R}}}\rightarrow 
S^{m'\mathbb{H}+ (n'+b^{+}_{2}(X))\widetilde{\mathbb{R}}}.
\end{equation}
Here $S^{1}_{+}$ denotes union of $S^{1}$ with a disjoint base point. To obtain a map between two spheres, we embed $S^{1}$ as the unit sphere $S(2\mathbb{R})$ in $\mathbb{R}^{2}$ and let 
$
c_{0}: S^{2\mathbb{R}}\rightarrow S^{\mathbb{R}}\wedge S(2\mathbb{R})_{+}
$ be the Pontragin-Thom collapsing map which collapses all the points outside a tubular 
neighborhood of $S(2\mathbb{R})$. Consider the composition 
\begin{equation}\label{eq: BF composition}
\begin{split}
S^{2\mathbb{R}}\wedge S^{(m'-\frac{\sigma(X)}{16})\mathbb{H}+n'\widetilde{\mathbb{R}}}\xrightarrow{c_{0}\wedge \id_{S^{(m'-\frac{\sigma(X)}{16})\mathbb{H}+n'\widetilde{\mathbb{R}}}}} S^{\mathbb{R}}&\wedge (S(2\mathbb{R})_{+})\wedge S^{(m'-\frac{\sigma(X)}{16})\mathbb{H}+n'\widetilde{\mathbb{R}}}\\ &\xrightarrow{\id_{S^{\mathbb{R}}}\wedge \widetilde{sw}^{+}} S^{\mathbb{R}}\wedge S^{m'\mathbb{H}+(n'+b^{+}_{2}(X))\widetilde{\mathbb{R}}}.
\end{split}    
\end{equation}
 It is a $\Pin$-equivariant map between two spheres and thus it defines an element in the $\Pin$-equivariant stable homotopy of spheres. This element is exactly the $\Pin$-equivariant family Bauer--Furuta invariant of $f$, denoted by $\FBF^{\Pin}(X,f,\widetilde{f})$. More precisely, we have $$\FBF^{\Pin}(X,f,\widetilde{f})\in\begin{cases}\{S^{\mathbb{R}+(-\frac{\sigma(X)}{16})\mathbb{H}},S^{b^{+}(X)\widetilde{\mathbb{R}}}\}^{\Pin}_{\mathcal{U}}& \mbox {if }\sigma(X)\leq 0\\
\{S^{\mathbb{R}},S^{b^{+}(X)\widetilde{\mathbb{R}}+\frac{\sigma(X)}{16}\mathbb{H}}\}^{\Pin}_{\mathcal{U}}& \mbox {if }\sigma(X)> 0
\end{cases}.$$
(See Definition \ref{def: equivariant stable homotopy group} for a precise description of these groups.) 

The $S^{1}$-equivariant Bauer--Furuta invariant $\FBF^{S^{1}}(X,f,\widetilde{f})$ and the non-equivariant Bauer--Furuta invariant $\FBF^{\{e\}}(X,f,\widetilde{f})$ are defined to be the restriction of $\FBF^{\Pin}(X,f,\widetilde{f})$ to the subgroups $S^{1},\{e\}\subset \Pin$ respectively. Therefore, we have 
$$
\FBF^{\Pin}(X,f,\widetilde{f})=0\implies \FBF^{S^{1}}(X,f,\widetilde{f})=0\implies \FBF^{\{e\}}(X,f,\widetilde{f})=0.
$$

All these invariants are unchanged under smooth isotopy. In particular, when $f$ is smoothly isotopic to $\id_{X}$ and $\widetilde{f}$ is the lift that corresponding to the identity map on the spin bundle, all of them are zero. To get an invariant of $f$, we can consider the unordered pair
$$
\{\FBF^{\Pin}(X,f,\widetilde{f}_{1}),\FBF^{\Pin}(X,f,\widetilde{f}_{2})\}
$$
for two possible lifts $\widetilde{f}_{1},\widetilde{f}_{2}$. If $f$ is smoothly isotopic to $\id_{X}$, then at least one element in this pair must be vanishing.
\begin{remark}
We briefly explain the motivation of the composition (\ref{eq: BF composition}). Since $S^{1}_{+}$ is stably homotopy equivalent to $S^{1}\vee S^{0}$, the map $\widetilde{sw}^{+}$ admits a splitting into $\widetilde{sw}^{+}_{0}\vee \widetilde{sw}_{1}^{+}$ in the stable homotopy category. The component $\widetilde{sw}^{+}_{0}$ represent the Bauer--Furuta invariant of the fiber $X$ and the component $\widetilde{sw}^{+}_{1}$ captures information about $f$. Since we are studying exotic diffeomorphisms on a fixed smooth manifold $X$, our main interest is $\widetilde{sw}_{1}^{+}$. The effect of composing with a suspension of $c_{0}$ in (\ref{eq: BF composition}) is exactly keeping $\widetilde{sw}_{1}^{+}$ and ignoring $\widetilde{sw}_{0}^{+}$. This also explains why the family Bauer--Furuta invariant can be trivial even when the Bauer--Furuta invariant of the fiber is nontrivial.  
\end{remark}
\subsection{The invariant of a loop of diffeomorphisms} Let $X$ be as before. Now we construct  the family Bauer--Furuta invariant of  a loop of diffeomorphisms on $X$. Let $\gamma: S^{1}\rightarrow \operatorname{Diff}(X)$ be a loop  based at $\operatorname{id}_{X}$. We use $\gamma$ as a clutching function and form a smooth fiber bundle 
$$
X\hookrightarrow T\gamma\rightarrow S^{2}.
$$
\begin{definition}\label{defi: spin-liftable}
$\gamma$ is \emph{spin-liftable} if the second Stiefel--Whitney class $w_{2}(T\gamma)$ vanishes.
\end{definition}

 Note that this condition is always satisfied if $\gamma$ is null-homotopic in the homeomorphism group of $X$. Because in this case, $T\gamma$ is homeomorphic to $X\times S^{2}$. By Wu's formula, the Seiefel--Whitney class is independent of the smooth structure, we have $w_{2}(T\gamma)=w_{2}(X\times S^{2})=0$.   

From now on, we assume $\gamma$ is spin-liftable. Then $T\gamma$ carries a family spin structure. Since the inclusion $X\hookrightarrow T\gamma$ induces an isomorphism
on $H^{1}(-;\mathbb{Z}/2)$, such a family spin structure is unique if we require that its restriction the fiber is the given spin structure on $X$. (This is different from the previous case of a single diffeomorphism.) 

Now we repeat the same construction as before and get a map 
\begin{equation}\label{eq: SW+ for loop}
\widetilde{sw}^{+}: S^{2}_{+}\wedge   S^{(m'-\frac{\sigma(X)}{16})\mathbb{H}+  n'\widetilde{\mathbb{R}}}\rightarrow 
S^{m'\mathbb{H}+ (n'+b^{+}_{2}(X))\widetilde{\mathbb{R}}}.
\end{equation}
for some $m',n'\gg 0$. We embed $S^{2}$ as the unit sphere in $S^{3\mathbb{R}}$ and form the composition between (suspensions of) $\widetilde{sw}^{+}$ and the Pontryagin-Thom collapsing map $S^{3\mathbb{R}}\rightarrow S^{\mathbb{R}}\wedge S(3\mathbb{R})_{+}$. This gives us a map 
$$
S^{3\mathbb{R}+(m'-\frac{\sigma(X)}{16})\mathbb{H}+  n'\widetilde{\mathbb{R}}}\rightarrow S^{\mathbb{R}+ m'\mathbb{H}+ (n'+b^{+}_{2}(X))\widetilde{\mathbb{R}}}.
$$
The $\Pin$-equivariant family Bauer--Furuta  invariant $
\FBF^{\Pin}(X,\gamma)$ is defined as the equivariant  stable  homotopy class of  this map. More  precisely, we have 
$$\FBF^{\Pin}(X,\gamma)\in\begin{cases}\{S^{2\mathbb{R}+(-\frac{\sigma(X)}{16})\mathbb{H}},S^{b^{+}(X)\widetilde{\mathbb{R}}}\}^{\Pin}_{\mathcal{U}}& \mbox {if }\sigma(X)\leq 0\\
\{S^{2\mathbb{R}},S^{b^{+}(X)\widetilde{\mathbb{R}}+\frac{\sigma(X)}{16}\mathbb{H}}\}^{\Pin}_{\mathcal{U}}& \mbox {if }\sigma(X)> 0
\end{cases}.$$
The invariants $\FBF^{S^{1}}(X,\gamma)$ and $\FBF^{\{e\}}(X,\gamma)$ are defined by restricting $\FBF^{\Pin}(X,\gamma)$ to the subgroups $\{e\},S^{1}$ respectively. Again, these invariants are unchanged under any smooth isotopies.
\begin{remark}
As noted in \cite{KM2020}, there are extra complications if one tries to generalize this construction and define invariants of elements in $\pi_{k}(\operatorname{Diff}(X))$ with $k\geq 2$. Because the symplectic group $\operatorname{Sp}(n)$ is 2-connected, every quaternion bundle over $S^{k+1}$ has a trivialization which is unique up to homotopy when $k\leq 1$. This allows us to define the canonical trivilizations of the bundles involved in construction (see (\ref{eq: trivialization})). However, this argument breaks down when $k\geq 2$.
\end{remark}
\section{Exotic Surfaces and Stabilization}\label{proof of exoticness}
\subsection{Exotic surfaces that survive a stabilization}
We will now prove Theorem \ref{surface}, which shows the existence of exotic surfaces that remains exotic after a single stabilization. The proof here relies on Theorem \ref{thm: generalized vanishing theorem} in the case of $S^{2}\times S^{2}$. Theorem \ref{thm: generalized vanishing theorem} will be proved in Section \ref{section: properties of FBF} in full generality. 

We start with the following technical lemma.

\begin{lemma}\label{lem: identity on neighborhood}
Let $D=\bigsqcup D^{2}\hookrightarrow X$ be a smooth, proper embedding of a finite disjoint union of disks into a 4-manifold $X$ with boundary. Consider a self-diffeomorphism $f:X\rightarrow X$ that fixes all points in a neighborhood of $\partial X$ and all points on $D$. Then $f$ is smoothly isotopic (relative to $\partial X$) to some $f'$ which fixes all points in a neighborhood of $(\partial X)\cup D$.
\end{lemma}
\begin{proof}
Let $N$ be the normal bundle of $D$ in $X$. By choosing a trivialization of this bundle, we get an identification $N=D\times \mathbb{R}^{2}$. 
Since $f$ fixes $D$ pointwisely, its differential induces a bundle automorphism $l': N\rightarrow N$. Let $\rho: N\rightarrow \nu(D)$ be a diffeomorphism to a tubular neighborhood of $D$. We use $l$ to denote the composition  $$\rho\circ l'\circ \rho^{-1}:\nu(D)\rightarrow \nu(D).$$For $\epsilon>0 $, we use $\nu_{\epsilon}(D)$ to denote $\rho(D\times D_{\epsilon})$, where $D_{\epsilon}\subset \mathbb{R}^{2}$ is the the ball centered at the origin with radius $\epsilon$. We pick $\epsilon$ small enough such that $f(\nu_{\epsilon}(D))\subset \nu(D)$. 
We first construct a smooth isotopy $h_{1}:[0,1]\times \nu_{\epsilon}(D)\rightarrow \nu(D)$ between the maps
$$f\mid_{\nu_{\epsilon}(D)}: \nu_{\epsilon}(D)\rightarrow \nu(D)
\quad \text{and}\quad l\mid_{\nu_{\epsilon}(D)}: \nu_{\epsilon}(D)\rightarrow \nu(D).
$$
(Isotopy here means that for any $t$, the map $h_{1}(t,-)$ is a diffeomorphism onto its image.) This is done via the standard rescaling trick: For any $t>0$, we use $\phi_{t}:\nu(D)\rightarrow \nu(D)$ to denote the self-diffeomorphism defined by $$\phi_{t}(\rho(x,\vec{v}))=\rho(x,t\vec{v})\quad \forall x\in D, \vec{v}\in \mathbb{R}^{2}.$$ 
Then the isotopy $h_{1}$ is given by 
$$
h_{1}(t,y):=\begin{cases} \phi^{-1}_{t}\circ f\circ \phi_{t}(t,y)&\mbox{if } t\in (0,1]\\
l(y)&\mbox{if } t=0
\end{cases}.
$$
Next, we construct an smooth isotopy $h_{2}: [0,1]\times \nu_{\epsilon}(D)\rightarrow \nu(D)$ from the map $l\mid_{\nu_{\epsilon}(D)}$ to the inclusion map $i: \nu_{\epsilon}(D)\rightarrow \nu(D)$. This is done via a simple obstruction argument: The map $l'$ is linear on each fiber so it has the form 
$$
l'(x,\vec{v})=(x,\tau(x)\cdot \vec{v})
$$
for some $\tau: D\rightarrow \mathrm{GL}^{+}(\mathbb{R}^{2})$ that equals $1$ near $\partial D$. Since $\mathrm{GL}^{+}(\mathbb{R}^{2})$ is homotopy equivalent to $S^{1}$ and $\pi_{2}(S^{1})=0$, there exists a smooth homotopy (relative to $\partial D$)
$$
\widetilde{\tau}: [0,1]\times D\rightarrow \mathrm{GL}^{+}(\mathbb{R}^{2})
$$
from $\tau$ to the constant map $1$. we define $\widetilde{l}': [0,1]\times N\rightarrow N$ by
$$
\widetilde{l}'(t,x,\vec{v})=(x,\widetilde{\tau}(t,x)\cdot \vec{v})\quad\text{ for }t\in [0,1], x\in D, \vec{v}\in \mathbb{R}^{2}.
$$
Then we define 
$
h_{2}(t,y):=\rho(\widetilde{l}'(t,\rho^{-1}(y))).
$ By compositing $h_{1}$ with $h_{2}$, we obtain a smooth isotopy between the maps $$f\mid_{\nu_{\epsilon}(D)}: \nu_{\epsilon}(D)\rightarrow \nu(D)
\quad \text{and}\quad  i:\nu_{\epsilon}(D)\rightarrow \nu(D).$$ Furthermore, this isotopy is trivial near $\partial X$. Therefore, we can apply the isotopy extension theorem to produce an isotopy of $f$ to some $f'$ that fixed all points in a neighborhood of $(\partial X)\cup D$.
\end{proof}  
 Next, we briefly recall the construction of the Dehn twist on the punctured $K3$. We pick an essential loop 
\begin{equation}\label{loop}
\gamma:[0,1]\rightarrow  SO(4)
\end{equation}
such that $\gamma(t)=1$ for $t$ near $0, 1$. Then we define the  self-diffeomorphism 
$$
\widetilde{\gamma}: [0,1]\times S^{3}\rightarrow [0,1]\times S^{3}
$$
by $\widetilde{\gamma}(t,x)=(t,\gamma(t)x)$. Let $K3^{\circ}$ be the punctured $K3$. We fix a diffeomorphism from $[0,1]\times S^{3}$ to a collar neighborhood of  $K3^{\circ}$ and define the Dehn twist as follows
\begin{equation}\label{dehn twist}
\delta_{K3^{\circ}}:= \widetilde{\gamma}\cup \operatorname{id}_{K3^{\circ}\setminus ([0,1]\times S^{3})}: K3^{\circ}\rightarrow  K3^{\circ}.\end{equation}

Now we are ready to show the existence of exotic surfaces that survive even after a single stabilization.\\
\begin{proof}[Proof of Theorem~\ref{surface}]

Recall that one can construct $K3^{\circ}$ by starting with $S^3\times I$ and then adding a bunch of 2-handles along $S^3\times \{1\}$ and cap it off with a 4-ball. Let $L$ be the set of attaching spheres on $S^3$. Then $L$ is a link with $22$ components. Let $D_{L}$ be the union
$$
(L\times I)\cup_{L\times \{1\}}\{\text{cores of the two handles}\}. 
$$ Then $D_{L}$ is a slice disk bounded by $L\subset \partial (K3^{\circ})$. 

Consider the Dehn-twist $\delta_{K3^{\circ}}$ defined in (\ref{dehn twist}). In \cite{KM2020}, Kronheimer--Mrowka showed that $\delta_{K3^{\circ}}$ is not smoothly isotopic to $\id_{K3^{\circ}}$ relative to the boundary.  Let $D'_L= \delta_{K3^{\circ}}(D_L)$, then $D'_{L}$ and  $D_L$ are topologically isotopic relative to the boundary, since the result of Krannich--Kupers \cite[Remark 6]{krannich2021torelli} and Orson--Powell \cite{OP} implies $\delta_{K3^{\circ}}$ is indeed topologically isotopic to $\id_{K3^{\circ}}$ relative to the boundary. Also it is obvious now that their complements are diffeomorphic.  

The claim is that $D_L$ and $D'_L$ are not smoothly isotopic even after a single stabilization. Suppose this were not true. Denote $\delta \# \id_{S^2\times S^2}$ by $\delta_{K3^{\circ}\#(S^2\times S^2)}$. By the isotopy extension theorem, there exists a self-diffeomorphism $g_{1}: K3^{\circ }\#(S^{2}\times S^{2})\rightarrow K3^{\circ }\#(S^{2}\times S^{2})$ that is smoothly isotopic to the identity relative to the boundary and sends $D_{L}$ to $D'_{L}$. Then $g_{1}\circ \delta_{K3^{\circ}\#(S^2\times S^2)}$ fixes all points on $D_{L}$ and is smoothly isotopic to $\delta_{K3^{\circ}\#(S^2\times S^2)}$ relative to the boundary. By Lemma \ref{lem: identity on neighborhood}, we can further isotope $g_{1}\circ \delta_{K3^{\circ}\#(S^2\times S^2)}$ relative to the boundary to a diffeomorphism $$g_{2}:K3^{\circ}\#(S^2\times S^2)\rightarrow K3^{\circ}\#(S^2\times S^2)$$ that is identity on the neighbourhood of $D_L \cup S^3\times \{0\}$. Notice that the complement of this set in $K3^{\circ}\#(S^{2}\times S^{2})$ is a punctured $S^2\times S^2$. Thus $g_{2}$ can be decompose as $\id_{K3^{\circ}}\# f$ for some $f \in \Diff(S^2\times S^2)$ which acts trivially on the cohomology. Furthermore, let $B$ be the small disk in $S^{2}\times S^{2}$ that we apply the connected sum construction. Then $f$ fixes $B$ pointwisely. 

So we have shown that $\delta_{K3^{\circ}\#S^{2}\times S^{2}}$ is smoothly isotopic to $\id_{K3^{\circ}}\# f$.
This further implies that the Dehn twist  
$$
\delta_{K3\# K3\# (S^{2}\times S^{2})}:K3\# K3\# (S^{2}\times S^{2})\rightarrow K3\# K3\# (S^{2}\times S^{2}) 
$$
along the neck between the two $K3's$ is smoothly  isotopic to $(\id_{K3\#K3})\#f$ .

Let $\id_{\s_{K3\# K3}}$ be the identity map on the spin bundle over $K3\# K3$. There are two lifts of $f$ to the spin bundle over $S^{2}\times S^{2}$. We set $\widetilde{f}$ to be the one that is identity on all the fibers over $B$. Then we can form the connected sum $\id_{\s_{K3\# K3}}\# \widetilde{f}$, which is a lift of $\id_{K3\#K3}\#f$ to an automorphism on the spin bundle. 

It follows the connected sum formula \cite[Proposition 2.26]{lin} that \begin{align*}
\FBF^{\Pin}(K3\#K3\#S^2\times S^2,\id_{K3\#K3}\#f, \id_{\s_{K3\# K3}}\#\widetilde{f})\\
= BF^{\Pin}(K3\#K3) \wedge \FBF^{\Pin}(S^2 \times S^2,f, \widetilde{f}). 
\end{align*}
But Theorem~\ref{thm: generalized vanishing theorem} implies that  $\FBF^{\Pin}(S^2 \times S^2,f, \widetilde{f})=0$, and hence 
\[
\FBF^{\Pin}(K3\#K3\#S^2\times S^2,\id_{K3\#K3}\#f, \id_{\s_{K3\#K3}}\#\widetilde{f})=0.
\]
But in \cite[Proof of Theorem~1.2]{lin}, the first author proves that
\[
\FBF^{\Pin}(K3\#K3\#S^2\times S^2,\delta_{K3\# K3\# (S^{2}\times S^{2})},
\widetilde{\delta}) \neq 0
\]
for any lift $\widetilde{\delta}$ of $\delta_{K3\# K3\# (S^{2}\times S^{2})}$ to the spin structure of $K3\#K3\#S^2\times S^2$. \footnote{There is a subtle difference between the setting of the current paper and \cite{lin}. Here we are using an incomplete universe $\mathbb{H}^{\infty}\oplus \widetilde{\mathbb{R}}^{\infty}\oplus \mathbb{R}^{\infty}$ to set up an equivariant stable homotopy category, while in \cite{lin}, a complete universe is used. (See Section \ref{section homotopy theory} for further explanations.) This difference does not cause a problem because by definition, $\FBF^{\Pin}(K3\#K3\#S^2\times S^2,\delta_{K3\# K3\# (S^{2}\times S^{2})},
\widetilde{\delta})\neq 0$ in a complete universe implies the same result in any universe.} Since $\delta_{K3\# K3\# (S^{2}\times S^{2})}$ is smoothly isotopic to $\id_{K3\#K3}\#f$, this contracts the isotopy invariance of the family Bauer--Furuta invariant. Thus $D_L$ and $D'_L$ are not smoothly isotopic relative to the boundary in $K3^{\circ}\# S^2\times S^2.$
\end{proof}
\subsection{An upper bound on the number of stabilizations} Now we move on to the proof of Proposition \ref{upper bound on stabilizations}, which gives an upper bound on the stabilizations needed to make $D_L$ and $D'_L$ smoothly isotopic. We start with part (1).
\begin{proof}[Proof of Proposition \ref{upper bound on stabilizations} (1)] 
Since $L_{0}$ is an unlink, there is a diffeomorphism $f:S^{3}\rightarrow S^{3}$ that is smoothly isotopic to $\id_{S^{3}}$ and sends $L_{0}$ to the standard unlink 
$$
L'_{0}=\{(\cos \theta,\sin\theta, i)\mid \theta\in [0,2\pi], i=1,\cdots m\}\subset \mathbb{R}^{3}\cup \{\infty\}=S^{3}.
$$
By taking the rotations along the $z$-axis, we can construct an essential loop $\gamma':[0,1]\rightarrow SO(4)$ such that $\gamma'(t)$ fixes $L'_{0}$ as a subset for any $t$. Then we consider the self-diffeomorphism
$$
\widetilde{\gamma}': [0,1]\times S^{3}\rightarrow [0,1]\times S^{3}
$$
defined by 
$$
\widetilde{\gamma}'(t,x):=(t,f^{-1}\circ \gamma'(t)\circ f(x)).
$$
By extending $\tilde{\gamma}'$ with the identity map on $K3^{\circ}\setminus ([0,1]\times S^{3})$, we obtain a self-diffeomorphism $\delta'_{K3^{\circ}}$ on $K3^{\circ}$. Since $\gamma'$ is homotopic to the loop $\gamma$ in (\ref{loop}), and $f$ is smoothly isotopic to $\id_{S^{3}}$, we see that $\delta'_{K3^{\circ}}$ is smoothly isotopic to $\delta_{K3^{\circ}}$ relative to the boundary. Therefore, $\delta'_{K3^{\circ}}(D_{L_{0}})$ and $D'_{L_{0}}$ are smoothly isotopic relative to the boundary. Furthermore, since $f^{-1}\circ \gamma'(t)\circ f$ fixes $L_0$ as a subset, the embedded surfaces $\delta'_{K3^{\circ}}(D_{L_0})$ and $D_{L_0}$ differ from each other by a Dehn twist on the domain $\bigsqcup_{m} D^{2}$. Since this Dehn twist is smoothly isotopic the identity map relative to  the boundary, the surfaces $D_{L_0}$ and $\delta'_{K3^{\circ}}(D_{L_0})$ are also smoothly isotopic relative to the boundary.  
\end{proof}
Figure \ref{K3}  is a Kirby diagram of $K3$ \cite[Page 305]{GompfStipsicz}. The attaching link has $21$ unknotted components and a single knotted component (which is a trefoil). By doing a handle slide of this knotted component along one of the unknoted components, we can change the middle crossing and make it into the unknot. This gives an example of $L$ whose all components are unknots. By Proposition \ref{upper bound on stabilizations} (1), all components of $D_{L}$ are smoothly isotopic to the corresponding components of $D'_{L}$ relative to the boundary. 

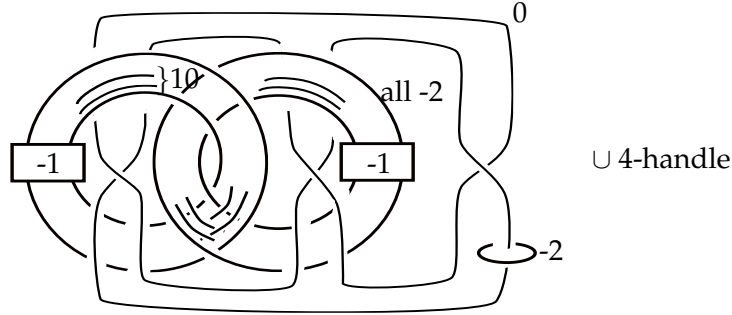
\begin{figure}[htbp]
\begin{center}
    \begin{overpic}[width=7cm, height=4cm]{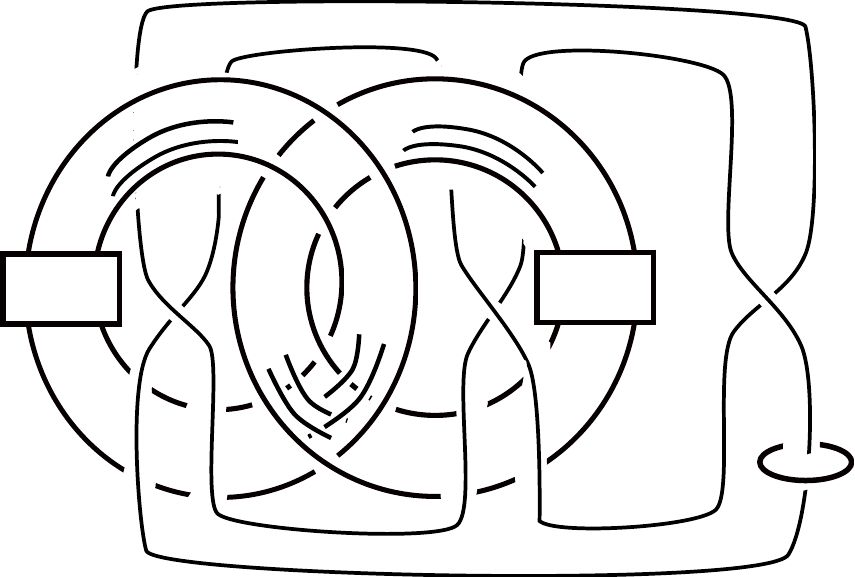}
    \put(10,53){-1}
    \put(135,53){-1}
    \put(140,80){all -2}
    \put(190,110){0}
    \put(200,20){-2}
    \put(55,85){$\}$10}
    \put(220,55){$\cup$ 4-handle}
    \end{overpic}
    \caption{A Kirby Diagram of K3.}\label{K3}
\end{center}

\end{figure}

Now we move on to  the proof of Proposition \ref{upper bound on stabilizations} (2). We start by the following lemma.
\begin{lemma}\label{lem: homology after removing the tubular neighborhood} Given a sublink $L'$ of $L$, let $\nu(D_{L'})$ be an open tubular neighborhood of $D_{L'}$ in $K3^{\circ}$. Then we have the following results. 
\begin{enumerate}
    \item [(i)] $K3^{\circ}\setminus \nu(D_{L'})$ and $K3^{\circ}\setminus \nu(D'_{L'})$ are simply connected.
    \item [(ii)] For any $n\geq 0$, the inclusion $(K3^{\circ}\setminus \nu(D_{L'})\#n (S^{2}\times S^{2})\hookrightarrow K3^{\circ}\#n(S^{2}\times S^{2})$ induces an injection on the second homotopy group $\pi_{2}(-)$.
\end{enumerate}
\end{lemma}
\begin{proof}
(i) We prove the case of $D_{L'}$ and the other case is similar. For each component $K$ of $L'$, let $m_{k}\subset \partial (K3^{\circ}\setminus \nu(D_{L'}))$ be the meridian of the embedded disk $D_{K}$.  
Since $K3^{\circ}$ is simply-connected, we see that $\pi_{1}(K3^{\circ}\setminus \nu(D_{L'}))$ is normally generated by $\{m_{K}\}$. So it suffices to  show that $m_{K}$ is null-homotopic in $K3^{\circ}\setminus \nu(D_{L'})$. To prove this, we consider the cocore of the 2-handle attached to $K$. This cocore becomes a cylinder in $K3^{\circ}\setminus \nu(D_{L'})$ and it provides a homotopy between $m_{K}$ and a knot on the boundary of the 4-handle $D^{4}$ (which doesn't intersect $\nu(D_{L'})$). Since $D^{4}$ is contractible, this knot (and hence $m_{K}$) is null-homotopic.

(ii) By the Mayer-Vietoris sequence, the map $$H_{2}((K3^{\circ}\setminus \nu(D_{L'}))\#n(S^{2}\times S^{2});\mathbb{Z})\rightarrow H_{2}(K3^{\circ}\#n(S^{2}\times S^{2});\mathbb{Z})$$ induced by the inclusion is injective. By the Hurewicz theorem, this map is also the induced map on $\pi_{2}(-)$. 
\end{proof}
The next proposition is a variation of the main theorem in \cite{AKMRS}.
\begin{proposition}\label{prop: one is enough}
Let $X$ be a simply-connected, smooth 4-manifold with boundary. Let $D_{K}$ and $D'_{K}$ be two smoothly embedded disks bounded by the same knot $K\subset \partial X$. Suppose the following conditions are all satisfied:
\begin{enumerate}
    \item [(i)] $X\setminus D_{K}$ and $X\setminus D'_{K}$ are simply-connected.
     \item [(ii)] $D_{K}$ and $D'_{K}$ are homotopic relative to  the boundary. 
     \item [(iii)] $[D_{K}]\in H_{2}(X,\partial X;\mathbb{Z}/2)$ is ordinary (i.e. it is not the  Poincar\'e dual of $w_{2}(X)$).
\end{enumerate}
Then in $X\# (S^{2}\times S^{2})$, $D_{K}$ and $D'_{K}$ are smoothly isotopic relative to the boundary.  
\end{proposition}
\begin{proof} The proof is adapted from the corresponding proof in \cite{AKMRS}. For completeness, we sketch the argument, with an emphasize on the adaptions needed here. First, since $X\setminus D_{K}$ is simply connected, we can find an immersed sphere $\Sigma$ that intersects $D_{K}$ transversely at a single point. Using condition (iii), one  can modify $\Sigma$ so that it has an even self-intersection number. Since $D'_{K}$ is homotopic to $D_{K}$ relative to the boundary, the algebraic intersection number $\Sigma\cdot D'_{K}$ equals $1$. After suitable finger moves and the (immersed) Whitney moves, we can isotope $D'_{K}$ and $\Sigma$ so that $\Sigma$ intersects both $D_{K}$ and $D'_{K}$ transversely at a single point. To arrange for $\Sigma$ to be embedded, we take a connected sum with $S^{2}\times S^{2}$ away from $D_{K}\cup D'_{K}\cup \Sigma$. Take coordinate 2-spheres $S=S^{2}\times \{\text{pt}\}$ and $T=\{\text{pt}\}\times S^{2}$. By replacing $\Sigma$ with its connected sum with $S$ along a tube, we may assume that $\Sigma$ has a  geometric dual $T$ (i.e., $T$ intersects transverse $\Sigma$ at a single point). Then we can eliminate the double points of $\Sigma$ by the Norman trick, tubing to parallel copies of $T$. This gives us an embedded $\Sigma$ with even self-intersection number. After tubing additional copies of $T$, we can arrange the self-intersection number of $\Sigma$ to be zero. Therefore,  $D_{K}$ and $D'_{K}$ have a common geometric dual $\Sigma$ in $X\# (S^{2}\times S^{2})$. No essential changes are necessary until step. Then we apply the light bulb theorem for embedded disks \cite[Corollary 3.3]{schwartz2021}, which states that $D_{K}$ and $D'_{K}$ are smoothly isotopic relative to the boundary if the following three conditions are verified:
\begin{enumerate}
    \item [(1)] $D_{K}$ and $D'_{K}$ are homotopic in $X\#(S^{2}\times S^{2})$ relative to the boundary. 
    \item [(2)] The Dax invariant $\operatorname{Dax}(D_{K},D'_{K})$ vanishes.
    \item [(3)] The map $\pi_{1}((X\#(S^{2}\times S^{2}))\setminus \Sigma)\rightarrow \pi_{1}(X\#(S^{2}\times S^{2}))$ induced by the inclusion is an isomorphism. (I.e. $(X\#(S^{2}\times S^{2}))\setminus \Sigma$ is simply connected.)
\end{enumerate}
(1) follows from our assumption (ii). (2) is trivial since any surfaces in a simply connected 4-manifold have trivial Dax invariant. (3) is satisfied because $\Sigma$ has an geometric dual $T$ (so the meridian of $\Sigma$ bounds a disk in $X\setminus \Sigma$). This finishes the proof.
\end{proof}
\begin{proof}[Proof of Proposition \ref{upper bound on stabilizations} (2)] Let $L_0\subset L$ be the $m$-component unlink. Denote the components of $L\setminus L_0$ by $K_{1}, K_{2},\cdots, K_{22-m}$. 
For $1\leq n\leq 22-m$, we set 
$$
L_{n}=L_{0}\bigsqcup K_{1}\bigsqcup \cdots \bigsqcup K_{n} \quad\text{for}\quad 1\leq n\leq  22-m.
$$
We will prove the following claim by induction. 
\begin{claim}
For any all $0\leq  n\leq 22-m$, $D_{L_{n}}$ and $D'_{L_{n}}$ are smoothly isotopic relative to the boundary in $K3^{\circ}\#n(S^{2}\times S^{2})$ .
\end{claim}
The case $n=0$ is implied by Proposition \ref{upper bound on stabilizations} (1). Now suppose the case $n=i$ has been proved. Denote $K3^{\circ}\#i(S^{2}\times S^{2})$ by $X_{i}$. Then by our induction hypothesis, there exists a diffeomorphism $f:X_{i}\rightarrow X_{i}$ which is smoothly isotopic to $\operatorname{id}_{X_{i}}$ relative to the boundary and satisfies $f(D'_{L_{i}})=D_{L_{i}}$. Note that $D_{K_{i+1}}$ and $f(D'_{K_{i+1}})$ are disks embedded in $X_{i}\setminus \nu(D_{L_{i}})$ bounded by the same knot $K_{i+1}$. To prove the case $n=i+1$, it suffices to show that $D_{K_{i+1}}$ and $f(D'_{K_{i+1}})$ are smoothly isotopic relative to  the boundary in $(X_{i}\setminus \nu(D_{L_{i}}))\# (S^{2}\times S^{2})$. (The connected sum is taken near $\partial X_{i}$ so that the isotopy between $f$ and $\id_{X_{i}}$ can be trivially extended.) By Proposition \ref{prop: one is enough}, we just need to check the following three conditions:
\begin{enumerate}
    \item $X_{i}\setminus \nu(D_{L_{i}})$, $X_{i}\setminus \nu(D_{L_{i}}\bigsqcup D_{K_{i+1}})$ and $X_{i}\setminus \nu(D_{L_{i}}\bigsqcup f(D'_{K_{i+1}}))$ are all simply connected. 
    \item The homology class $[D_{K_{i+1}}]\in H_{2}(X_{i}\setminus \nu(D_{L_{i}});\mathbb{Z}/2)$ is ordinary.
    \item $D_{K_{i+1}}$ and $f(D'_{K_{i+1}})$ are homotopic in $X_{i}\setminus \nu(D_{L_{i}})$ relative to the boundary. 
\end{enumerate}
Note that  $$X_{i}\setminus \nu(D_{L_{i}}\bigsqcup D_{K_{i+1}})\cong (K3^{\circ} \setminus \nu(D_{L_{i+1}}))\#i(S^{2}\times S^{2}),$$
$$
X_{i}\setminus \nu(D_{L_{i}}\bigsqcup f(D'_{K_{i+1}}))\cong (K3^{\circ} \setminus \nu(D'_{L_{i+1}}))\#i(S^{2}\times S^{2}).
$$
Therefore, (1) is directly implied by Lemma \ref{lem: homology after removing the tubular neighborhood} (i).

To prove (2), we consider the cocore of the two-handle $h_{2}^{i+1}$ attached to $K_{i+1}$. The boundary of this cocore (i.e., the  belt sphere) is a knot on the boundary of the 4-handle $D^{4}$. Hence it bounds a properly embedded surface $F$ in $D^{4}$. We consider the surface 
$$
F'=\{\text{cocore of }h_{2}^{i+1}\}\cup F\hookrightarrow   K3^{\circ}\setminus\nu(D_{L_{i}})\subset K3^{\circ}.
$$
Since $K3^{\circ}$ has an even intersection form, the self-intersection number of $F'$ is even. Note that  
$$
D_{K_{i+1}}\cdot F'=1\not\equiv F'\cdot F'\mod 2.
$$
So $[D_{K_{i+1}}]$ is ordinary.

To prove (3), we note that  $D_{K_{i+1}}$ and $D'_{K_{i+1}}$ are topologically isotopic relative to the boundary in $K3^{\circ}$. Hence they are also topologically isotopic relative to the boundary in $X_{i}$. Since $D'_{K_{i+1}}$ is smoothly isotopic to $f(D'_{K_{i+1}})$ relative to the boundary, we see that $D_{K_{i+1}}$ and $f(D'_{K_{i+1}})$ are topologically isotopic relative to the boundary in $X_{i}$. In particular, they are homotopic to each other relative to the  boundary in $X_{i}$. Consider the map 
$
\tau: S^{2}\rightarrow X_{i}
$
which sends the upper hemisphere and the lower hemisphere to $D_{K_{i+1}}$ and $f(D'_{K_{i+1}})$ respectively. Since $D_{K_{i+1}}$ and $f(D'_{K_{i+1}})$ are homotopic in $X_{i}$ relative to the boundary, $\tau(S^{2})$ is null-homotopic in $X_{i}$. By Lemma \ref{lem: homology after removing the tubular neighborhood} (2), $\tau(S^{2})$ is also null-homotopic in $X_{i}\setminus \nu(D_{L_{i}})$. Therefore, 
$D_{K_{i+1}}$ and $f(D'_{K_{i+1}})$ are homotopic in $X_{i}\setminus \nu(D_{L_{i}})$ relative to the boundary. This proves (3).
\end{proof}

\section{Equivariant Stable Homotopy Theory}\label{section homotopy theory}
In this section, we briefly recall the definitions of equivariant stable homotopy groups and state some important properties about them. In particular, we state a version of the tom Dieck splitting theorem for an incomplete universe, which turns the computation of the equivariant stable group of a $G$-space into the computation of certain non-equivariant stable homotopy groups. Then we use this theorem to prove some algebraic propositions about the $\Pin$-equivariant stable homotopy groups of spheres. These results will be essentially used in Section \ref{section: properties of FBF} to prove the vanishing theorem for the family Bauer--Furuta invariant. Our exposition here follows the classical method of Adams \cite{Adams1982} (which assumes $G$ is finite). We refer to  \cite{schwede2014,Lewis-May-Steinberger1986} for the general case of compact Lie groups and more complete discussions using the language of equivariant spectra.

\subsection{The  tom Dieck splitting theorem} 
Let $G$ be a compact Lie group. We consider an infinite dimensional $G$-representation $\mathcal{U}$ that satisfies the following two properties:
\begin{itemize}
    \item There is a canonical embedding from the infinite dimensional trivial representation $ \bigoplus\limits_{\infty}\mathbb{R}$ to $\mathcal{U}$. 
    \item For any finite dimensional subrepresentation $V$ of $\mathcal{U}$, there exists an embedding from  $\bigoplus\limits_{\infty}V$ to
    $\mathcal{U}$.  
\end{itemize}
We call such $\mathcal{U}$ a ``$G$-universe''. If all finite dimensional representation of $G$ can embedded into $\mathcal{U}$, then we call $\mathcal{U}$ a complete universe. Otherwise, we say $\mathcal{U}$ is incomplete. For our topological applications, it is necessary to consider incomplete universes.

We equip $\mathcal{U}$ with an $G$-invariant inner product. For any pointed $G$-spaces $X,Y$ and any finite dimensional subrepresentations $V\subset W\subset \mathcal{U}$, there is a map 
$$
[S^{V}\wedge X,S^{V}\wedge Y]^{G}\xrightarrow{\id_{S^{W-V}}\wedge -} [S^{W}\wedge X,S^{W}\wedge Y]^{G},
$$
where $\id_{S^{W-V}}$  denotes the identity map on the representation sphere of the orthogonal  complement of $V$ in $W$. 
\begin{definition}\label{def: equivariant stable homotopy group}
For any pointed spaces $X,Y$, the equivariant stable homotopy group is defined as 
\begin{equation*}
\{X,Y\}^{G}_{\mathcal{U}}:=\lim\limits_{\substack{\longrightarrow\\ V\subset \mathcal{U} }} [S^{V}\wedge X,S^{V}\wedge Y]^{G},
\end{equation*}
where the limit is taken over all finite dimensional $G$-invariant subspace $V\subset \mathcal{U}$. (Just like the non-equivariant setting, $\{X,Y\}^{G}_{\mathcal{U}}$ has a abelian group structure.) For $k\in \mathbb{N}$, we use $\pi^{G}_{k,\mathcal{U}}(X)$ to denote $\{S^{k}, X\}^{G}_{\mathcal{U}}$. \end{definition}
 When $G$ is the trivial group, there is essentially a unique choice of $\mathcal{U}$. In this case, we denote $\{X,Y\}^{G}_{\mathcal{U}}$ by $\{X,Y\}$ and denote  $\pi^{G}_{k,\mathcal{U}}(X)$ by $\pi^{s}_{k}(X)$. 

Given a closed subgroup $H$ of $G$, we use $\operatorname{res}^{G}_{H}(\mathcal{U})$ to denote the same space $\mathcal{U}$, regarded it as an $H$-representation space. Then there is a well-defined restriction map 
$$
\operatorname{res}^{G}_{H}: \{X,Y\}^{G}_{\mathcal{U}}\rightarrow \{X,Y\}^{H}_{\operatorname{res}^{G}_{H}(\mathcal{U})}
$$
defined by only remembering the $H$-action. 

Also, let $N_{G}H$ be the normalizer of $H$ in $G$ and let $W_{G}H=N_{G}H/H$ be the Weyl group. Then the $H$-fixed point of $\mathcal{U}$, which we denote by $\mathcal{U}^{H}$, is an $W_{G}H$-universe. There is a well-defined map 
\begin{equation}
 (-)_{Y,X,G}^{H}: \{Y,X\}^{G}_{\mathcal{U}}\rightarrow \{Y^{H},X^{H}\}^{W_{G}H}_{\mathcal{U}^{H}},   
\end{equation}
which sends a  maps to its restriction to the $H$-fixed points. In the special case $Y=S^{k}$, we have the map 
$$
(-)_{S^{k},X,G}^{H}:\pi^{G}_{k,\mathcal{U}}(X)\rightarrow \pi^{W_{G}H}_{k,\mathcal{U}^{H}}(X^{H}).
$$
When $X,Y,G$ are obvious from the context, we will suppress them from our notations and simply write $(-)^{H}$. 

Next, for any pointed $G$-space $X$, we introduce a map
\begin{equation}\label{eq: zeta}
\zeta^{G}_{H}: \pi^{W_{G}H}_{k, \mathcal{U}^{H}}(EW_{G}H_{+}\wedge X^{H})\rightarrow \pi^{G}_{k,\mathcal{U}}(X)
\end{equation}
which is crucial in the tom Dieck splitting theorem. The precise definition of $\zeta^{G}_{H}$ can be found in \cite[Page 60]{schwede2014}. For our purpose, it suffices to consider the case that $H$ is a normal subgroup of $G$. In this case, we have $N_{G}H=G, W_{G}H=G/H$ and the map $\zeta^{G}_{H}$ has a simple description: We have taken an element $[f]\in \pi^{G/H}_{k, \mathcal{U}^{H}}((EG/H)_{+}\wedge X^{H})$, represented by a $G/H$-equivariant map 
$$
g: S^{V}\wedge S^{k}\rightarrow S^{V}\wedge (EG/H)_{+}\wedge X^{H}
$$
for some $G/H$-representation $V\subset \mathcal{U}^{H}\subset \mathcal{U}$. We regard $f$ as a $G$-equivariant map and consider the composition 
\begin{equation}\label{eq: definition of zeta}
\widetilde{g}:S^{V}\wedge S^{k}\xrightarrow{f} S^{V}\wedge (EG/H)_{+}\wedge X^{H}\xrightarrow{1\wedge p_{0}\wedge 1} S^{V}\wedge S^{0}\wedge X^{H}= S^{V}\wedge X^{H}\xrightarrow{1\wedge i} S^{V}\wedge X.
\end{equation}

Here $p_{0}: (EG/H)_{+}\rightarrow S^{0}$ is the map that preserves the base point and collapses the whole $EG/H$ to the other point in $S^{0}$, and the map $i:X^{H}\rightarrow X$ is the inclusion map. Note that both $p_{0}$ and $i$ are $G$-equivariant because $H$ is normal. We define $\zeta^{G}_{H}[g]:=[\widetilde{g}]$. 

Now we state the tom Dieck
splitting theorem for an incomplete universe. This theorem was originally proved in \cite{tomDieck1975} for a complete universe. The version given below is a special case of \cite[Theorem 3.3]{Lewis2000}. See \cite{blumberg2021equivariant,Lewis-May-Steinberger1986} for various generalizations to equivariant spectra. 

\begin{theorem}\label{thm: tom Dieck} Let $X$ be a pointed $G$-space. Then the map 
\begin{equation}\label{eq: tom Dieck iso}
\sum_{[H]}\zeta^{G}_{H}:\bigoplus_{[H]}\pi^{W_{G}H}_{k,\mathcal{U}^{H}}(EW_{G}H_{+}\wedge X^{H})\rightarrow \pi^{G}_{k,\mathcal{U}}(X)
\end{equation}
is an isomorphism for any $k\geq 0$. Here the direct sum is taken over all conjugacy classes of subgroups $H$ that arises as the isotropy group of some point in $\mathcal{U}$.
\end{theorem}

Note that the $W_{G}H$-action on $(EW_{G}H)_{+}\wedge X^{H}$ is free away from the base point. By the Adams' isomorphism theorem  \cite[Corollary 7.2]{Lewis-May-Steinberger1986}, we have an isomorphism 
\begin{equation}\label{eq: Adams' isomorphism}
\pi^{s}_{k}((S^{\operatorname{Ad}(W_{G}H)}\wedge (EW_{G}H)_{+}\wedge X^{H})/W_{G}H)\xrightarrow{\cong }\pi^{W_{G}H}_{k,\mathcal{U}^{H}}((EWH)_{+}\wedge X^{H}).
\end{equation}
Here $\operatorname{Ad}(W_{G}H)$ denotes the adjoint representation of $W_{G}H$ on its Lie algebra. Note this isomorphism relies on  the fact that $W_{G}H$ can be embedded into $\mathcal{U}^{H}$ because $H$ is an isotropy group. (See \cite[Theorem 12.2]{Lewis2000}.) Using (\ref{eq: Adams' isomorphism}), we can rewrite the isomorphism (\ref{eq: tom Dieck iso}) as the isomorphism 
\begin{equation}\label{eq: tom Dieck iso 2}
\bigoplus_{[H]}\pi^{s}_{k}((S^{\operatorname{Ad}(W_{G}H)}\wedge (EW_{G}H)_{+}\wedge X^{H})/W_{G}H)\xrightarrow{\cong }\pi^{G}_{k,\mathcal{U}}(X).
\end{equation}

\subsection{Some additional lemmas}
We introduce a few additional lemmas that will be useful in our later computations. 
\begin{lemma}\label{lem: adjunction}
Let $X$ be a pointed space and $Y$ be a pointed $G$-space. Then $\{G_{+}\wedge X, Y\}^{G}_{\mathcal{U}}\cong \{X,Y\}$.
\end{lemma}
\begin{proof}
This comes from the standard adjunction $(G_{+}\wedge - )\rightleftharpoons \operatorname{res}^{G}_{\{e\}}(-)$. (See \cite[Theorem 4.7]{Lewis-May-Steinberger1986} and \cite[Theorem 5.1]{Adams1982}).
\end{proof}

\begin{lemma}\label{lem: simple vanishing result}
Let $X$ be a $d$-dimensional $G$-CW complex and let $Y$ a pointed $G$-space. Suppose the $G$-action on $X$ is free away from the base ponint and that the non-equivariant stable homotopy group $\pi^{s}_{m}(Y)$ is trivial for any $0\leq m\leq d$. Then  we have $\{X,Y\}^{G}_{\mathcal{U}}\cong 0$.
\end{lemma}
\begin{proof} 
This follows from a standard induction argument. Recall that a $G$-CW complex is obtained by attaching the $G$-cells $(G/H)\times D^{k}$ for various closed subgroups $H$ inductively over $k$. Since the $G$-action on $X$ is free away from the base point, all cells (except the base point) are of the form $G\times D^{k}$. Let $X_{k}$ be the $k$-skeleton of $X$. Then $$X_{k}/X_{k-1}=G_{+}\wedge (\mathop{\vee}\limits_{A(k)}  S^{k\mathbb{R}})$$ where $A(k)$ denotes the set of $k$-cells in $X$. For any $k\leq d$, by Lemma \ref{lem: adjunction}, we have 
$$
\{G_{+}\wedge (\mathop{\vee}\limits_{A(k)}  S^{k\mathbb{R}}),Y\}^{G}_{\mathcal{U}}\cong \{ \mathop{\vee}\limits_{A(k)}  S^{k\mathbb{R}},Y\}\cong 0.
$$
Therefore, we can inductively prove that $\{X_{k},Y\}^{G}_{\mathcal{U}}\cong 0$ using the long exact sequence 
$$
\cdots \rightarrow \{X_{k}/X_{k-1},Y\}^{G}_{\mathcal{U}}\rightarrow \{X_{k},Y\}^{G}_{\mathcal{U}}\rightarrow \{X_{k-1},Y\}^{G}_{\mathcal{U}}\rightarrow \cdots.
$$
induced by the cofiber sequence $X_{k-1}\rightarrow X_{k}\rightarrow X_{k}/X_{k-1}$.
\end{proof}

\begin{lemma}\label{lem: trivial bundle}
Let $X$ be a space with a free $G$-action and let $V$ be a $n$-dimensional $G$-representation. Consider the vector bundle $\widetilde{V}:=(V\times X)/G$ over the space $X/G$. Suppose $\widetilde{V}$ is a trivial bundle. Then there is a base-point preserving, $G$-equivariant homeomorphism from  $S^{V}\wedge X_{+}$ to $S^{n}\wedge X_{+}$.
\end{lemma}
\begin{proof}
Since $\widetilde{V}$ is a trivial bundle, there exists a map $\rho_{0}: (V\times X)/G\rightarrow \mathbb{R}^{n}$ which restricts to a linear isomorphism on any fiber. We compose this map with the quotient map $V\times X\rightarrow (V\times X)/G$ to get an $G$-equivariant map $\rho_{1}: V\times X\rightarrow \mathbb{R}^{n}$. Then we consider the $G$-equivariant homeomorphism 
$$
\rho: V\times X\rightarrow \mathbb{R}^{n}\times X
$$
defined by $\rho(\vec{v},x):= (\rho_{1}(\vec{v},x),x)$. After adding the base point at infinity, $\rho$ induces the homeomorphism from $S^{V}\wedge X_{+}$ to $S^{n}\wedge X_{+}$.
\end{proof}

\subsection{$\Pin$-equivariant stable homotopy groups}
From now on, we set $\mathcal{U}$ to be the incomplete $\Pin$-universe  $\mathbb{H}^{\infty}\oplus \widetilde{\mathbb{R}}^{\infty}\oplus \mathbb{R}^{\infty}$. From this universe, we can define four other universes 
$\operatorname{res}^{\Pin}_{S^{1}}(\mathcal{U})$, $\operatorname{res}^{\Pin}_{\{e\}}(\mathcal{U})$, $\mathcal{U}^{S^{1}}$ and $\mathcal{U}^{\Pin}$ .  We denote them by $\mathcal{U}_{1}$, $\mathcal{U}_{2}$, $\mathcal{U}_{3}$ and $\mathcal{U}_{4}$ respectively. The corresponding Lie groups are $S^{1}, \{e\}, C_{2}$ and $\{e\}$ respectively. Here $C_{2}=\Pin/S^{1}$ denotes the cyclic group of order 2. 

We will use $S(\mathbb{H}^{\infty})$ as a specific model for $E\Pin$. Then $B\Pin$ is the quotient of 
$BS^{1}=\mathbb{CP}^{\infty}$ under the involution 
$$
[(z_{0},z_{1},z_{2},z_{3},\cdots )]\rightarrow [(-\overline{z_{1}},\overline{z_{0}},-\overline{z_{3}},\overline{z_{2}},\cdots)]
.$$
We start with the following lemma:

\begin{lemma}\label{lem: computation of the kernel 1}
Let $X$ be a $\Pin$-CW complex. Then for any $k\geq 0$ the kernel of the map 
\begin{equation}\label{eq: S1-fixed point for Pin-space}
(-)_{S^{k},X,\Pin}^{S^{1}}: \pi^{\Pin}_{k,\mathcal{U}}(X)\rightarrow \pi^{C_{2}}_{k,\mathcal{U}_{3}}(X)
\end{equation}
is exactly the image of the injective map $\zeta^{\Pin}_{\{e\}}$. In particular, the kernel of the map (\ref{eq: S1-fixed point for Pin-space}) is isomorphic to the non-equivariant stable homotopy group $\pi_{k}^{s}((S^{\widetilde{\mathbb{R}}}\wedge E\Pin_{+}\wedge X)/\Pin)$.
\end{lemma}
\begin{proof}
We consider the following diagram

\begin{equation*}\label{diagram: tom Dieck}
\xymatrix{
\pi^{\Pin}_{k,\mathcal{U}}(E\Pin_{+}\wedge X)\oplus \pi^{C_{2}}_{k,\mathcal{U}_{3}}((EC_{2})_{+}\wedge X^{S^{1}})\oplus\pi_{k,\mathcal{U}_{4}}(X^{\Pin}) \ar[d]^{\pi} \ar[rrr]^-{\zeta^{\Pin}_{\{e\}}\oplus \zeta^{\Pin}_{S^{1}}\oplus \zeta^{\Pin}_{\Pin}} & & & \pi^{\Pin}_{k,\mathcal{U}}(X) \ar[d]^{(-)_{S^{k},X,\Pin}^{S^{1}}}\\
 \pi^{C_{2}}_{k,\mathcal{U}_{3}}((EC_{2})_{+}\wedge X^{S^{1}})\oplus\pi_{k,\mathcal{U}_{4}}(X^{\Pin})\ar[rrr]^-{\zeta^{C_{2}}_{C_{2}}\oplus \zeta^{C_{2}}_{\{e\}}} & & &\pi^{C_{2}}_{k,\mathcal{U}_{3}}(X^{S^{1}})},    
\end{equation*}
where $\pi$ is the projection to the second and third summand. By checking the explicit description of $\zeta^{G}_{H}$ in (\ref{eq: definition of zeta}), it is straightforward to see that this diagram commutes. Therefore, the kernel of $(-)_{S^{k},X,\Pin}^{S^{1}}$ is exactly the image of $\zeta^{G}_{\{e\}}$. We then apply (\ref{eq: Adams' isomorphism}) to finish the proof.
\end{proof}
Using the same argument, one can prove the following lemma. 
\begin{lemma}\label{lem: kernel of S1 fixed points 2}
Let $X$ be an $S^{1}$-CW complex. Then the kernel of the map $$(-)_{S^{k},X,S^{1}}^{S^{1}}: 
 \pi^{S^{1}}_{k,\mathcal{U}_{3}}(X)\rightarrow \pi^{s}_{k}(X^{S^{1}})
$$
equals the image of the injective map $\zeta^{S^{1}}_{\{e\}}: \pi^{S^{1}}_{k,\mathcal{U}_{3}}(ES^{1}_{+}\wedge X)\rightarrow \pi^{S^{1}}_{k,\mathcal{U}_{3}}(X)$. In particular, the kernel is isomorphic to $\pi^{s}_{k-1}((ES^{1}_{+}\wedge X)/S^{1})$.
\end{lemma}
The following propositions (Proposition \ref{prop: algebraic vanishing result1}$\sim$ Proposition \ref{prop: algebraic vanishing result3}) are the key results in this section. They are the homotopy theoretic version of Theorem \ref{thm: generalized vanishing theorem}.
\begin{proposition}\label{prop: algebraic vanishing result1} Suppose $m,n\geq 0$ and are not both zero. Then the map $$(-)_{S^{1},S^{m\mathbb{H}+n\widetilde{\mathbb{R}}},\Pin}^{S^{1}}: \pi^{\Pin}_{1,\mathcal{U}}(S^{m\mathbb{H}+n\widetilde{\mathbb{R}}})\rightarrow \pi^{C_{2}}_{1,\mathcal{U}_{3}}(S^{n\widetilde{\mathbb{R}}})
$$ is injective.
\end{proposition}
\begin{proposition}\label{prop: algebraic vanishing result2} 
Suppose $n>4m\geq 0$. Then the map  
$$(-)^{S^{1}}_{S^{\mathbb{R}+m\mathbb{H}},S^{n\widetilde{\mathbb{R}}},\Pin}:\{S^{\mathbb{R}+m\mathbb{H}},S^{n\widetilde{\mathbb{R}}}\}^{\Pin}_{\mathcal{U}}\rightarrow \pi^{C_{2}}_{1,\mathcal{U}_{3}}(S^{n\widetilde{\mathbb{R}}})$$ is injective. 
\end{proposition}
\begin{proposition}\label{prop: algebraic vanishing result3}
Suppose $n=4m\geq 0$. Then the map 
$$(-)_{S^{\mathbb{R}+m\mathbb{H}},S^{n\widetilde{\mathbb{R}}},\Pin}^{S^{1}}: \{S^{\mathbb{R}+m\mathbb{H}},S^{n\widetilde{\mathbb{R}}}\}^{\Pin}_{\mathcal{U}}\rightarrow \pi^{C_{2}}_{1,\mathcal{U}_{3}}(S^{n\widetilde{\mathbb{R}}})
$$
has a kernel isomorphic to $\mathbb{Z}/2$. Furthermore, let $\alpha$ be the nonzero element in this kernel. Then we have $\operatorname{res}_{S^{1}}^{\Pin}(\alpha)=0$.
\end{proposition}
\begin{proposition}\label{prop: algebraic vanishing result4} The map 
$$(-)_{S^{2\mathbb{R}},S^{0},\Pin}^{S^{1}}: \pi^{\Pin}_{2,\mathcal{U}}(S^{0})\rightarrow \pi^{C_{2}}_{2,\mathcal{U}_{3}}(S^{0})
$$
has a kernel isomorphic to $\mathbb{Z}/2$. Furthermore, let $\alpha'$ be the nonzero element in this kernel. Then we have $\operatorname{res}_{S^{1}}^{\Pin}(\alpha')=0$.
\end{proposition}
\begin{proof}[Proof of Proposition \ref{prop: algebraic vanishing result1}] By Lemmma \ref{lem: computation of the kernel 1}, we have 
\begin{equation}\label{eq: ker to nonequivriant}
\ker (-)^{S^{1}}_{S^{1},S^{m\mathbb{H}+n\widetilde{\mathbb{R}}},\Pin}\cong \pi^{s}_{1}((S^{m\mathbb{H}+(n+1)\widetilde{\mathbb{R}}}\wedge E\Pin_{+})/\Pin).
\end{equation}

Let $\xi_{\widetilde{\mathbb{R}}}$ (resp. $\xi_{\mathbb{H}}$) the be bundle associated to the universal bundle $\Pin\hookrightarrow E\Pin\rightarrow B\Pin$ and the representation $\mathbb{R}$ (resp. $\mathbb{H}$). Then we have  $$(S^{m\mathbb{H}+(n+1)\widetilde{\mathbb{R}}}\wedge E\Pin_{+})/\Pin=\operatorname{Th}(B\Pin,m\xi_{\mathbb{H}}\oplus (n+1)\xi_{\widetilde{\mathbb{R}}}).$$ Here $\operatorname{Th}(-)$ denotes the Thom space. 

Suppose $m,n\geq 0$ are not both zero. Then 
$m\xi_{\mathbb{H}}\oplus (n+1)\xi_{\widetilde{\mathbb{R}}}$ is a vector bundle of dimension $\geq 2$. Therefore, other than the base point, the space $\operatorname{Th}(B\Pin,m\xi_{\mathbb{H}}\oplus (n+1)\xi_{\widetilde{\mathbb{R}}})$ has no cells in dimension $0$ and $1$. By the skeletal approximation theorem, we have  
$$\ker (-)^{S^{1}}_{S^{1},S^{m\mathbb{H}+n\widetilde{\mathbb{R}}},\Pin}\cong \pi^{s}_{1}(\operatorname{Th}(B\Pin,m\xi_{\mathbb{H}}\oplus (n+1)\xi_{\widetilde{\mathbb{R}}}))\cong 0.$$\end{proof}

\begin{proof}[Proof of Proposition \ref{prop: algebraic vanishing result2}] The case $m=0$ has been covered by Proposition \ref{prop: algebraic vanishing result1} so we assume $m\geq 1$, which implies that  $n\geq 5$. Take $$\beta\in \ker ((-)_{S^{\mathbb{R}+m\mathbb{H}},S^{n\widetilde{\mathbb{R}}},\Pin}^{S^{1}}).$$ Take the inclusion $i_{0}: S^{\mathbb{R}}\rightarrow S^{\mathbb{R}+m\mathbb{H}}$. We also use $i_{0}$ to denote the corresponding element in  $\{S^{\mathbb{R}},S^{\mathbb{R}+m\mathbb{H}}\}^{\Pin}_{\mathcal{U}}$.  Then we have 
\begin{equation}
(-)_{S^{\mathbb{R}},S^{n\widetilde{\mathbb{R}}},\Pin}^{S^{1}}(\beta\cdot i_{0})=(-)_{S^{\mathbb{R}+m\mathbb{H}},S^{n\widetilde{\mathbb{R}}},\Pin}^{S^{1}}(\beta)\cdot (-)_{S^{\mathbb{R}},S^{\mathbb{R}+m\mathbb{H}},\Pin}^{S^{1}}(i_{0})=0. 
\end{equation} 
But $\ker ((-)_{S^{\mathbb{R}},S^{n\widetilde{\mathbb{R}}},\Pin}^{S^{1}})$    
is trivial by (1). Hence we see that $\beta\cdot i_{0}=0$. 

Next, we consider the long exact sequence 
$$
\cdots\rightarrow \{S^{2\mathbb{R}}\wedge (S(m\mathbb{H})_{+}),S^{n\widetilde{\mathbb{R}}}\}^{\Pin}_{\mathcal{U}}\rightarrow \{S^{\mathbb{R}+m\mathbb{H}},S^{n\widetilde{\mathbb{R}}}\}^{\Pin}_{\mathcal{U}}\xrightarrow{-\cdot i_{0} }  \{S^{\mathbb{R}},S^{n\widetilde{\mathbb{R}}}\}^{\Pin}_{\mathcal{U}}\rightarrow \cdots
$$
induced by the cofiber sequence $S^{0}\rightarrow S^{m\mathbb{H}}\rightarrow S^{\mathbb{R}}\wedge S(m\mathbb{H})_{+}$. Note that $\Pin$ acts freely on $S^{2\mathbb{R}}\wedge (S(m\mathbb{H})_{+})$ away from the base point, and the quotient is a CW complex of dimension $4m$. Since $n>4m$, by Lemma \ref{lem: simple vanishing result}, we see that $\{S^{2\mathbb{R}}\wedge (S(m\mathbb{H})_{+}),S^{n\widetilde{\mathbb{R}}}\}^{\Pin}_{\mathcal{U}}\cong 0$. Hence the map $-\cdot i_{0}$ is injective. Since $\beta\cdot i_{0}=0$, we get $\beta=0$.
\end{proof}

We need some preparations in order to prove Proposition \ref{prop: algebraic vanishing result3}.
\begin{lemma}\label{lem: gamma is trivial when restriected to S1}
(1) For $m>0$, we have $\{S^{2\mathbb{R}}\wedge S(m\mathbb{H})_{+},S^{4m\widetilde{\mathbb{R}}}\}^{\Pin}_{\mathcal{U}}\cong \mathbb{Z}/2$. 
(2) Let $\gamma$ be the unique nonzero element in $\{S^{2\mathbb{R}}\wedge S(m\mathbb{H})_{+},S^{4m\widetilde{\mathbb{R}}}\}$. Then $\operatorname{res}_{S^{1}}^{\Pin}(\gamma)=0$. 
\begin{proof} (1) The $\Pin$-action on $S^{2\mathbb{R}}\wedge S((m-1)\mathbb{H})_{+}$ is free away from the base point and the quotient is a CW-complex of dimension $4m-4$. By Lemma \ref{lem: simple vanishing result}, we have 
$$
\{S^{2\mathbb{R}}\wedge S((m-1)\mathbb{H})_{+},S^{4m\widetilde{\mathbb{R}}}\}^{\Pin}_{\mathcal{U}}\cong \{S^{3\mathbb{R}}\wedge S((m-1)\mathbb{H})_{+},S^{4m\widetilde{\mathbb{R}}}\}^{\Pin}_{\mathcal{U}}\cong 0.
$$
Consider the cofiber sequence $$S((m-1)\mathbb{H})_{+}\hookrightarrow S(m\mathbb{H})_{+}\rightarrow S(m\mathbb{H})/S((m-1)\mathbb{H}).$$ Since $S(m\mathbb{H})/S((m-1)\mathbb{H})=S^{(m-1)\mathbb{H}}\wedge (S(\mathbb{H})_{+})$, we have the induced long exact sequence of homotopy groups $$\{S^{2\mathbb{R}}\wedge S(m\mathbb{H})_{+},S^{4m\widetilde{\mathbb{R}}}\}^{\Pin}_{\mathcal{U}}\cong \{S^{2\mathbb{R}+(m-1)\mathbb{H}}\wedge (S(\mathbb{H})_{+}),S^{4m\widetilde{\mathbb{R}}}\}^{\Pin}_{\mathcal{U}}.$$
To compute this group, we consider the vector bundles $\xi'_{\mathbb{H}}$ (resp. $\xi'_{4\widetilde{\mathbb{R}}}$) associated to the principal bundle $$\Pin\hookrightarrow S(\mathbb{H})\rightarrow \mathbb{RP}^{2}$$ for the representations $\mathbb{H}$ (resp.  $4\widetilde{\mathbb{R}}$). Since the structure group of both bundles can be lifted to $SU(2)$ and any $SU(2)$ bundles is trivial on a 2-dimensional CW complex. We see that $\xi'_{\mathbb{H}}$ and $\xi'_{4\widetilde{\mathbb{R}}}$ are both trivial. By Lemma \ref{lem: trivial bundle}, we have the following $\Pin$-equivariant homeomorphisms 
$$
S^{4\mathbb{R}+(m-1)\mathbb{H}}\wedge (S(\mathbb{H})_{+})\cong S^{4m\mathbb{R}}\wedge (S(\mathbb{H})_{+})\cong S^{4m\widetilde{\mathbb{R}}}\wedge (S(\mathbb{H})_{+}).
$$
They give the following isomorphism 
\begin{equation*}
\begin{split}
\{S^{2\mathbb{R}+(m-1)\mathbb{H}}\wedge (S(\mathbb{H})_{+}),S^{4m\widetilde{\mathbb{R}}}\}^{\Pin}_{\mathcal{U}}&\cong  \{S^{4\mathbb{R}+(m-1)\mathbb{H}}\wedge (S(\mathbb{H})_{+}),S^{4m\widetilde{\mathbb{R}}+2\mathbb{R}}\}^{\Pin}_{\mathcal{U}} \\ &\cong \{S^{4m\widetilde{\mathbb{R}}}\wedge (S(\mathbb{H})_{+}),S^{4m\widetilde{\mathbb{R}}+2\mathbb{R}}\}^{\Pin}_{\mathcal{U}}\\
&\cong \{S(\mathbb{H})_{+}, S^{2\mathbb{R}} \}^{\Pin}_{\mathcal{U}}\\
&\cong \{\mathbb{RP}^{2}_{+},S^{2}\} \\&\cong \mathbb{Z}/2.
\end{split}
\end{equation*}
(2) The element $\operatorname{res}^{\Pin}_{S^{1}}(\gamma)$ belongs to the group $$\{S^{2\mathbb{R}}\wedge S(m\mathbb{H})_{+},S^{4m\widetilde{\mathbb{R}}}\}^{S^{1}}_{\mathcal{U}_{3}}\cong \{ S(\mathbb{H})_{+},S^{(4m-2)\widetilde{\mathbb{R}}}\}^{S^{1}}_{\mathcal{U}_{3}}.$$
Since $S^{1}$ acts freely on $S(m\mathbb{H})$ and trivially on $S^{(4m-2)\widetilde{\mathbb{R}}}$, we have 
$$
\{ S(m\mathbb{H})_{+},S^{(4m-2)\widetilde{\mathbb{R}}}\}^{S^{1}}_{\mathcal{U}_{3}}\cong \{ (S(m\mathbb{H})_{+})/S^{1},S^{(4m-2)\widetilde{\mathbb{R}}}\}\cong \{\mathbb{CP}^{2m-1}_{+},S^{4m-2}\}\cong \mathbb{Z}.
$$
Since $\gamma$ is 2-torsion, $\operatorname{res}^{\Pin}_{S^{1}}(\gamma)$ must be zero.
\end{proof}
\begin{lemma}\label{lem: q* injective}
For $m>0$, the map 
\begin{equation}
    q^{*}: \{S^{2\mathbb{R}}\wedge (S(m\mathbb{H})_{+}),S^{4m\widetilde{\mathbb{R}}}\}^{\Pin}_{\mathcal{U}}\rightarrow \{S^{\mathbb{R}+m\mathbb{H}}, S^{4m\widetilde{\mathbb{R}}}\}^{\Pin}_{\mathcal{U}}\end{equation}
    induced by the quotient map $q: S^{m\mathbb{H}}\rightarrow S^{m\mathbb{H}}/S^{0}=S^{\mathbb{R}}\wedge (S(m\mathbb{H})_{+})$ is injective. \end{lemma}
\end{lemma}
\begin{proof}
We consider the long exact sequence 
$$
\cdots \rightarrow\{S^{2\mathbb{R}},S^{4m\widetilde{\mathbb{R}}}\}^{\Pin}_{\mathcal{U}}\xrightarrow{p_{1}^{*}}\{S^{2\mathbb{R}}\wedge (S(m\mathbb{H})_{+}),S^{4m\widetilde{\mathbb{R}}}\}^{\Pin}_{\mathcal{U}}\xrightarrow{q^{*}} \{S^{\mathbb{R}+m\mathbb{H}}, S^{4m\widetilde{\mathbb{R}}}\}^{\Pin}_{\mathcal{U}}\rightarrow \cdots
$$
induced by the cofiber sequence $S(m\mathbb{H})_{+}\xrightarrow{p_{1}}S^{0}\xrightarrow{i_{1}} S^{m\mathbb{H}}$, where $i_{1}$ is the inclusion map and 
$p_{1}: S(m\mathbb{H})_{+}\rightarrow S^{0}$ is the map that collapses the whole $S(m\mathbb{H})$ to the non-base point in $S^{0}$. It suffices to show that the map $p_{1}^{*}$ is trivial. By (\ref{eq: tom Dieck iso 2}), we have
$$
\{S^{2\mathbb{R}},S^{4m\widetilde{\mathbb{R}}}\}^{\Pin}_{\mathcal{U}}\cong \{S^{2}, (S^{(4m+1)\widetilde{\mathbb{R}}}\wedge E\Pin_+)/\Pin\}\oplus \{S^{2}, (S^{4m\widetilde{\mathbb{R}}}\wedge (EC_2)_+)/C_{2}\} \oplus \{S^{2}, S^{0}\}.
$$
Since $m>0$, the first two summands vanishes by the skeletal approximation theorem. Hence we have $$\{S^{2\mathbb{R}},S^{4m\widetilde{\mathbb{R}}}\}^{\Pin}_{\mathcal{U}}\cong \{S^{2}, S^{0}\}\cong \mathbb{Z}/2,$$ 
and it generated by $\xi^{\Pin}_{\Pin}(\eta^{2})$. (Here $\eta\in \pi_{1}(S^{0})$ is the Hopf map.) We need to show that $p_{1}^{*}(\xi^{\Pin}_{\Pin}(\eta^{2}))=0$. By the definition of $\xi^{\Pin}_{\Pin}$ in (\ref{eq: definition of zeta}), $\xi^{\Pin}_{\Pin}(\eta^{2})$ is equal to the following composition of stable maps
$$
S^{2}\xrightarrow{\eta^{2}}S^{0}\xrightarrow{i_{1}} S^{4m\widetilde{\mathbb{R}}},
$$
where $i_{1}: S^{0}\rightarrow S^{4m\widetilde{\mathbb{R}}}$ is the inclusion. 
Therefore, in order to show that $p_{1}^{*}$ is zero, it suffices to show that the composition 
\begin{equation}\label{eq: composition of projection and inclusion}
S(m\mathbb{H})_{+}\xrightarrow{p_{1}} S^{0}\xrightarrow{i_{1}} S^{4m\widetilde{\mathbb{R}}}
\end{equation}
is stably null-homotopic. By the exact sequence 
\begin{equation}
\cdots \rightarrow\{S^{m\mathbb{H}},S^{4m\widetilde{\mathbb{R}}}\}^{\Pin}_{\mathcal{U}}\rightarrow \{S^{0},S^{4m\widetilde{\mathbb{R}}}\}^{\Pin}_{\mathcal{U}}\xrightarrow{p^{*}}\{S(m\mathbb{H})_{+},S^{4m\widetilde{\mathbb{R}}}\}^{\Pin}_{\mathcal{U}}\rightarrow \cdots,    
\end{equation}
the composition (\ref{eq: composition of projection and inclusion}) is trivial if $i_{1}: S^{0}\rightarrow S^{4m\widetilde{\mathbb{R}}}$ can be extended to a map $\tilde{i}: S^{m\mathbb{H}}\rightarrow S^{4m\widetilde{\mathbb{R}}}$. Note that there is a $\Pin$-equivariant map 
$
\varphi: S^{\mathbb{H}}\rightarrow S^{3\widetilde{\mathbb{R}}}
$ defined by $$\varphi(z+wj):=(|z|^{2}-|w|^{2},\operatorname{Re}(z\bar{w}), \operatorname{Im}(z\bar{w})).$$ By smashing $m$ copies of $\varphi$ and the inclusion map $S^{0}\hookrightarrow S^{m\widetilde{\mathbb{R}}}$, we get the desired extension $\widetilde{i}$.
\end{proof}
\begin{proof}[Proof of Proposition \ref{prop: algebraic vanishing result3}] We first consider the case $m=n=0$. By (\ref{eq: ker to nonequivriant}), we get 
$$\ker (-)_{S^{\mathbb{R}},S^{0},\Pin}^{S^{1}}\cong \pi^{s}_{1}(\operatorname{Th}(B\Pin,\xi_{\widetilde{\mathbb{R}}})).$$
Note that the 3-skeleton of $B\Pin$ is $\mathbb{RP}^{2}$ and the restriction of $\xi_{\widetilde{\mathbb{R}}}$ to the 3-skeleton is exactly the tautalogical line bundle over it. Therefore, the $4$-skeleton of 
$\operatorname{Th}(B\Pin,\xi_{\widetilde{\mathbb{R}}})$ is $\mathbb{RP}^{3}$. This implies that
$$
\ker (-)_{S^{1},S^{0},\Pin}^{S^{1}}\cong \pi^{s}_{1}(\mathbb{RP}^{3})\cong \mathbb{Z}/2.
$$
Let $\alpha$ be the nonzero element in this kernel. By definition, $\operatorname{res}_{S^{1}}^{\Pin}(\alpha)$ belongs to the kernel of the map 
$$(-)_{S^{1},S^{\widetilde{\mathbb{R}}},S^{1}}^{S^{1}}: \pi^{S^{1}}_{1,\mathcal{U}_{2}}(S^{0})\rightarrow \pi^{s}_{1}(S^{0})
.$$ By Lemma \ref{lem: kernel of S1 fixed points 2}, we have 
$$
\ker(-)_{S^{1},S^{\widetilde{\mathbb{R}}},S^{1}}^{S^{1}}\cong \pi^{s}_{0}(\mathbb{CP}^{\infty}_{+})\cong \mathbb{Z}.
$$
Since $\alpha$ is 2-torsion, we must have $\operatorname{res}_{S^{1}}^{\Pin}(\alpha)=0$. 

Now we consider the case $n=4m>0$. We take a element 
$$\beta\in \ker (-)_{S^{\mathbb{R}+m\mathbb{H}},S^{n\widetilde{\mathbb{R}}},\Pin}^{S^{1}}.$$
Similar to the proof of Proposition \ref{prop: algebraic vanishing result2}, we can show that $\beta$ belongs to the image of the map 
$$\{S^{2\mathbb{R}},S^{4m\widetilde{\mathbb{R}}}\}^{\Pin}_{\mathcal{U}}\xrightarrow{p_{1}^{*}}\{S^{2\mathbb{R}}\wedge (S(m\mathbb{H})_{+}),S^{4m\widetilde{\mathbb{R}}}\}^{\Pin}_{\mathcal{U}}.$$ By Lemma \ref{lem: gamma is trivial when restriected to S1}, $\{S^{2\mathbb{R}},S^{4m\widetilde{\mathbb{R}}}\}^{\Pin}_{\mathcal{U}}=\{0,\gamma\}$. So $\beta$ is either $0$ or $p_{1}^{*}(\gamma)$. By Lemma \ref{lem: q* injective}, $p_{1}^{*}(\gamma)\neq 0$. Also note that $(-)^{S^{1}}(\gamma)=0$ because the $S^{1}$-action on its domain $S^{2\mathbb{R}}\wedge S(m\mathbb{H})_{+}$ is free away from the base point. This implies that $p_{1}^{*}(\gamma)$ actually belongs to $\ker((-)_{S^{\mathbb{R}+m\mathbb{H}},S^{n\widetilde{\mathbb{R}}},\Pin}^{S^{1}})$. To this end, we see that 
$$
\ker((-)_{S^{\mathbb{R}+m\mathbb{H}},S^{n\widetilde{\mathbb{R}}},\Pin}^{S^{1}})=\{0, p_{1}^{*}(\gamma)\}\cong \mathbb{Z}/2.
$$
Since $\operatorname{res}^{\Pin}_{S^{1}}(\gamma)=0$ by Lemma \ref{lem: gamma is trivial when restriected to S1}, we have $\operatorname{res}^{\Pin}_{S^{1}}(p_{1}^{*}(\gamma))=0$.
\end{proof}
\begin{proof}[Proof of Proposition \ref{prop: algebraic vanishing result4}] By Lemma \ref{lem: computation of the kernel 1}, we have 
$$\ker (-)_{S^{2\mathbb{R}},S^{0},\Pin}^{S^{1}}\cong \pi^{s}_{2}(\operatorname{Th}(B\Pin,\xi_{\widetilde{\mathbb{R}}}))\cong  \pi^{s}_{2}(\mathbb{RP}^{3}).$$
Stably, $\mathbb{RP}^{3}$ splits into $S^{3}\vee \mathbb{RP}^{2}$. Note that $\mathbb{RP}^{2}$ fits in to the cofiber sequence 
$$
S^{1}\xrightarrow{\cdot 2} S^{1}\rightarrow \mathbb{RP}^{2} 
$$
Using the induced long exact sequence of stable homotopy groups, we get an isomorphism 
$$\ker (-)_{S^{2\mathbb{R}},S^{0},\Pin}^{S^{1}}\cong\pi^{s}_{2}(\mathbb{RP}^{3})\xrightarrow[\eta\cdot -]{\cong } \pi^{s}_{1}(\mathbb{RP}^{3})\cong \mathbb{Z}/2,$$
where $\eta\in \pi^{s}_{1}(S^{0})$ denotes the Hopf map. Let $\alpha'$ be the nonzero element in this kernel. Then since  the smash product with $\eta$ commutes with the Adams isomorphism (\ref{eq: Adams' isomorphism}), we have $\alpha'=\eta\cdot \alpha$, where $\alpha\in \ker (-)_{S^{\mathbb{R}},S^{0},\Pin}^{S^{1}}$ is the  element established in Proposition \ref{prop: algebraic vanishing result3}. By that proposition, we get
$$
\operatorname{res}^{\Pin}_{S^{1}}(\alpha')=\eta\cdot  \operatorname{res}^{\Pin}_{S^{1}}(\alpha)=0.$$
\end{proof}

\section{Properties of The Family Bauer--Furuta Invariant}\label{section: properties of FBF} 
In this section, we use the homotopy theoretic results that we proved in Section \ref{section homotopy theory} to study the family Bauer--Furuta invariant. 
We will start with proving the vanishing properties of the family Bauer--Furuta invariant. Then we give a space level construction of the stable homotopy class $\alpha_{0}$ given in Theorem \ref{theorem main: curious Pin(2)-element}, which provides a potential tool to detect exotic diffeomorphisms on $S^{4}$.

\subsection{Proof of the vanishing results.} To prove Theorem \ref{thm: generalized vanishing theorem}, we give the following definition, which is motivated by the construction of the family Bauer--Furuta map defined in Section \ref{section: definition of FBF}. 

\begin{definition}\label{definition: elements of BF-type} Let $\alpha$ be an element in one of the following stable homotopy groups:
\begin{itemize}
    \item $\{S^{k\mathbb{R}}\wedge S^{m\mathbb{H}},S^{n\widetilde{\mathbb{R}}}\}^{\Pin}_{\mathcal{U}}$ for some integers $k,m,n\geq 0$, or
    \item $\{S^{k\mathbb{R}},S^{n\widetilde{\mathbb{R}}+(-m)\mathbb{H}}\}^{\Pin}_{\mathcal{U}}$ for some integers $k,n\geq 0$ and $m<0$.
\end{itemize}
We say $\alpha$ is an element \emph{of BF-type} if it satisfies the following condition
\begin{equation*}
(-)^{S^{1}}(\alpha)=0\in \{S^{k\mathbb{R}},S^{n\widetilde{\mathbb{R}}}\}^{C_{2}}_{\mathcal{U}_{3}}.
\end{equation*}Namely, the restriction of $\alpha$ to the $S^{1}$-fixed points is trivial as a $C_{2}$-equivariant stable homotopy class.
\end{definition}

\begin{lemma}\label{lem: FSW is of BF-type} Let $X$ be a spin 4-manifold with $b_{1}(X)=0$. Then we have the  following results: 
\begin{enumerate}
    \item[(i)] Let $f$ be a orientation preserving self-diffeomorphism on $X$ that preserves the spin structure and acts trivially on $H^{*}(X;\mathbb{R})$. Let $\widetilde{f}$ be an auotomorphism of the spin structure on $X$ that lifts $f$. Then $\FBF^{\Pin}(X,f,\widetilde{f})$ is an element of BF-type. 
    \item[(ii)] Let $\gamma:S^{1}\rightarrow \operatorname{Diff}(X)$ be a loop based at $\operatorname{id}_{X}$. Suppose $\gamma$ is spin-liftable (see Definition \ref{defi: spin-liftable}). Then $\FBF^{\Pin}(X,\gamma)$ is an element of BF-type. 
\end{enumerate}
\end{lemma}
\begin{proof} We focus on the proof of (i) since (ii) is similar. Consider the approximated Seiberg--Witten map 
$\widetilde{sw}: U\rightarrow V$ (see equation (\ref{eq: sw})). Recall that both the domain and the target of this map consists of pairs $(\alpha,\phi)$, where $\alpha$ is a differential form on $X$ and $\phi$ is a spinor on $X$. The symmetry group $S^{1}$ is the constant gauge group. Hence it acts trivially on $\alpha$ and acts as the scalar multiplication on $\phi$. From this description, we see that $(\alpha,\phi)$ is an $S^{1}$-fixed point if and only if $\phi=0$. As a result, when restricted to the $S^{1}$-fixed points, $\widetilde{sw}$ is a finite dimensional approximation of the family of linear operators $$\{(d^{*},d^{+})_{s}: \Omega^{1}(X_{s};i\mathbb{R})\rightarrow (\Omega^{0}(X_{s};i\mathbb{R})/\mathbb{R})\oplus \Omega_{+}^{2}(X_{s};i\mathbb{R})\}_{s\in S(2\mathbb{R})}.$$
The trivilization (\ref{eq: trivialization}) induce isomorphisms 
$$
U\cap \Omega^{1}(X_{s};i\mathbb{R}) \cong \mathbb{R}^{n'},
$$
$$
V\cap (\Omega^{0}(X_{s};i\mathbb{R})/\mathbb{R})\oplus \Omega^{2}(X_{s};i\mathbb{R})\cong \mathbb{R}^{n'+b^{+}_{2}(X_{s})}.
$$
Furthermore, by checking the construction of these trivilizations (\ref{eq: trivialization}) in \cite{KM2020,lin}, we see that under the above isomorphisms,  the operator $(d^{*},d^{+})_{s}$ is exactly the inclusion $\mathbb{R}^{n'}\hookrightarrow \mathbb{R}^{n'+b^{+}_{2}(X)}$ into the first $n'$ components. From this discussion, we see that when we restrict the induced map 
$$\widetilde{sw}^{+}: S^{1}_{+}\wedge   S^{(m'-\frac{\sigma(X)}{16})\mathbb{H}+  n'\widetilde{\mathbb{R}}}\rightarrow 
S^{m'\mathbb{H}+ (n'+b^{+}_{2}(X))\widetilde{\mathbb{R}}}$$
to the $S^{1}$-fixed points, it is homotopic to the following map
\begin{equation}\label{eq: restriction to S1-fixed points}
(S(2\mathbb{R})_{+})\wedge S^{n'\widetilde{\mathbb{R}}} \xrightarrow{p_{2}\wedge i_{2}} S^{0}\wedge S^{(n+n')\widetilde{\mathbb{R}}}     
\end{equation}
where $p_{2}:S(2\mathbb{R})_{+} \rightarrow S^{0}$ preserves the base point and collapse $S(2\mathbb{R})$ to the other point in $S^{0}$, and $i_{2}:  S^{n'\widetilde{\mathbb{R}}}\rightarrow S^{(n+n')\widetilde{\mathbb{R}}}$ is the inclusion into the first $n'$ components. 

Recall that $\FBF^{\Pin}(X,f,\widetilde{f})$ is defined via the composition (\ref{eq: BF composition}). From our discussion above, the element $(-)^{S^{1}}(\FBF^{\Pin}(X,f,\widetilde{f}))$ can be represented by the following composition 
$$
S^{2\mathbb{R}}\wedge S^{n'\widetilde{\mathbb{R}}}\xrightarrow{c_{0}\wedge \id_{S^{n'\widetilde{\mathbb{R}}}}} S^{\mathbb{R}}\wedge (S(2\mathbb{R})_{+})\wedge S^{n'\widetilde{\mathbb{R}}}\xrightarrow{\id_{S^{\mathbb{R}}}\wedge p_2\wedge i_2} S^{\mathbb{R}}\wedge S^{0}\wedge S^{(n+n')\widetilde{\mathbb{R}}}. 
$$
Note that $c_{0}\circ (\id_{S^{\mathbb{R}}}\wedge p_{2}): S^{2\mathbb{R}}\rightarrow S^{\mathbb{R}}$ is null-homotopic. Thus  $(-)^{S^{1}}(\FBF^{\Pin}(X,f,\widetilde{f}))=0$.
\end{proof}
The following proposition is just a restatement of Proposition \ref{prop: algebraic vanishing result1} $\sim$ Proposition \ref{prop: algebraic vanishing result4}.

\begin{proposition}\label{prop: elements of BF-type}
\begin{itemize}
    \item[(i)] Suppose $m,n\geq 0$ and are not both zero. Then  $\pi^{\Pin}_{1,\mathcal{U}}(S^{m\mathbb{H}+n\widetilde{\mathbb{R}}})$ doesn't contain any nonzero element of BF-type. 
    \item[(ii)] Suppose $n>4m\geq 0$. Then $\{S^{\mathbb{R}+m\mathbb{H}},S^{n\widetilde{\mathbb{R}}}\}^{\Pin}_{\mathcal{U}}$ doesn't contain any nonzero element of BF-type. 
    \item[(iii)] Suppose $n=4m\geq 0$. Then  
$\{S^{\mathbb{R}+m\mathbb{H}},S^{n\widetilde{\mathbb{R}}}\}^{\Pin}_{\mathcal{U}}$ contains a unique nonzero element of BF-type, denoted by $\alpha_{0}$. Furthermore, we have $\operatorname{res}_{S^{1}}^{\Pin}(\alpha_{0})=0$.
\item [(iv)] $\pi^{\Pin}_{2}(S^{0})$ contains an unique nonzero element of BF-type, denoted by $\alpha_{1}$. Furthermore, we have $\operatorname{res}_{S^{1}}^{\Pin}(\alpha_{1})=0$.
 \end{itemize}
\end{proposition}
\begin{proof}[Proof of Theorem \ref{theorem main: vanishing result for loops}] Let $\gamma$ an exotic loop in $\operatorname{Diff}(S^{4})$. Then $\gamma$ is spin-liftable. By Lemma \ref{lem: FSW is of BF-type}, $\FBF^{\Pin}(S^{4},\gamma)$ is an element in  $\pi^{\Pin}_{2}(S^{0})$ of BF-type. By Proposition \ref{prop: elements of BF-type} (iv), $\FBF^{\Pin}(S^{4},\gamma)$ is either $0$ or $\alpha_{1}$. So we have 
$$
\FBF^{S^{1}}(S^{4},\gamma)=\operatorname{res}^{\Pin}_{S^{1}}(\FBF^{\Pin}(S^{4},\gamma))=0.
$$
This further implies that $
\FBF^{\{e\}}(S^{4},\gamma)=0$.

Now suppose $\gamma$ is homotopic to  $2\gamma'$ for some loop $\gamma'$ in $\operatorname{Diff}(S^{4})$. We may assume that $\gamma'$ is also spin-liftable by compositing it with an essential loop in $\operatorname{SO(5)}\subset \operatorname{Diff}(S^{4})$. By Proposition \ref{prop: elements of BF-type} (iv) again, $\FBF^{\Pin}(S^{4},\gamma')$ is either $0$ or $\alpha_{1}$. Since $\alpha_{1}$ is 2-torsion, we have 
$$
\FBF^{\Pin}(S^{4},\gamma)=2\FBF^{\Pin}(S^{4},\gamma')=0.$$
\end{proof}
\begin{proof}[Proof of Theorem \ref{theorem main: curious Pin(2)-element}] This is a special case Proposition \ref{prop: elements of BF-type} (iii) for $m=n=0$. 
\end{proof}

\begin{proof}[Proof of Theorem \ref{thm: generalized vanishing theorem}:] Assume the intersection form of $X$ is isomorphic to $2mE8\oplus n(\begin{smallmatrix}0,1\\
1,0
\end{smallmatrix})$ for some $m\in \mathbb{Z}$ and $n\in \mathbb{N}$. Then $2\chi(X)+3\sigma(X)=4+4\cdot (n-4m)$.

Suppose $2\chi(X)+3\sigma(X)=4$. Then $m=4n$. By Lemma \ref{lem: FSW is of BF-type}, $\FBF^{\Pin}(X,f,\widetilde{f})$ is an element in 
$\{S^{\mathbb{R}+m\mathbb{H}},S^{n\widetilde{\mathbb{R}}}\}^{\Pin}_{\mathcal{U}}$ of BF-type. Therefore, by Proposition \ref{prop: elements of BF-type} (iii), we have
$$
\FBF^{S^{1}}(X,f,\widetilde{f})=\operatorname{res}^{\Pin}_{S^{1}}(\FBF^{\Pin}(X,f,\widetilde{f}))=0.
$$ This proves (1).

Suppose $2\chi(X)+3\sigma(X)>4$. Then $n>4m$. Again by Lemma \ref{lem: FSW is of BF-type}, $\FBF^{\Pin}(X,f,\widetilde{f})$ is an element of BF-type in $\pi^{\Pin}_{1,\mathcal{U}}(S^{-m\mathbb{H}+n\widetilde{\mathbb{R}}})$ if $m<0$ and element of BF-type in $\{S^{\mathbb{R}+m\mathbb{H}},S^{n\widetilde{\mathbb{R}}}\}^{\Pin}_{\mathcal{U}}$ if $m\geq 0$. Then we apply Proposition \ref{prop: elements of BF-type} (i) and (ii) to conclude that $\FBF^{\Pin}(X,f,\widetilde{f})=0$. This proves (2).
\end{proof} 

\begin{remark}\label{remark: nonvanishing FBF} We mention various examples of nonvanishing family Bauer-Furuta invariants. These examples show that various conditions in Theorem \ref{thm: generalized vanishing theorem} are indeed necessary. 
\begin{itemize}
    \item Consider the Dehn twist $\delta_{K3\# K3}$ along the separating neck in $K3\# K3$. Then Kronheimer--Mrowka \cite{KM2020} prove that 
    $$
    \FBF^{\{e\}}(K3\#K3,\delta_{K3\# K3}, \widetilde{\delta})=\eta^{3}\neq 0
    $$
    for any spin bundle automorphism $\widetilde{\delta}$ which lifts $\delta_{K3\# K3}$. Here $\eta\in \pi^{s}_{1}(S^{0})$ denotes the Hopf element. Note that $2
    \chi(K3\#K3)+3\sigma(K3\#K3)=-4.$ 
    \item Consider the covering transformation on the spin bundle over $K3$. This is a lift of $\id_{K3}$, denoted by $\widetilde{\id}_{K3}$. Then Kronehimer--Mrowka \cite{KM2020} proves that  
    $$
    \FBF^{\{e\}}(K3, \id_{K3},\widetilde{\id}_{K3})=\eta^{2}\neq 0.
    $$
    Note that $2\chi(K3)+3\sigma(K3)=0$.
    \item Consider the covering transformation on the spin bundle over $K3\# (S^{2}\times S^{2})$, denoted by $\widetilde{\id}_{K3\# (S^{2}\times S^{2})}$. Then by \cite[Proof of Theorem 1.2]{lin}, we have 
    $$
    \FBF^{\Pin}(K3\#(S^{2}\times S^{2}), \id_{K3\#(S^{2}\times S^{2})},\widetilde{\id}_{K3\#(S^{2}\times S^{2})})\neq 0.
    $$
    Note that $2\chi(K3\# (S^{2}\times S^{2}))+ 3\sigma(K3\# (S^{2}\times S^{2}))=4$.
\end{itemize}
\end{remark}

\subsection{A space level description of $\alpha_{0}$} The  element $\alpha_{0}\in \pi^{\Pin}_{1,\mathcal{U}}(S^{0})$ given in Theorem \ref{theorem main: curious Pin(2)-element} is of special interest because it could arises as the $\Pin$-equivariant family Bauer--Furuta invariant of an exotic diffeomorphism on $S^{4}$. In this section, we explicitly construct a map 
$$
l: S^{\mathbb{R}+\mathbb{H}}\rightarrow S^{\mathbb{H}}
$$
which represents $\alpha_{0}$. 

We treat $\Pin$ as a submanifold of $\mathbb{H}$, with the standard embedding 
$$
\Pin=\{e^{ i\theta}\}\cup \{j e^{i\theta}\}\hookrightarrow \mathbb{H}.
$$
Denote the tangent bundle and the normal bundle of $\Pin$ by $T_{*}\Pin$ and $N_{*}\Pin$ respectively. Both bundles are $\Pin$-bundles. Namely, there are $\Pin$-actions on both bundles that cover the left-multiplication action of $\Pin$ on itself. Furthermore, we have trivilizations 
\begin{equation}\label{eq: trivliazation1}
T_{*}\Pin\cong \mathbb{R}\times \Pin, 
\end{equation}
\begin{equation}\label{eq: trivliazation2}
T_{*} \Pin \oplus N_{*}\Pin\cong \mathbb{H}\times \Pin.
\end{equation}
Here (\ref{eq: trivliazation1}) is induced by the Lie group structure on $\Pin$ and (\ref{eq: trivliazation2}) is restricted from the trivialization of $T_{*}\mathbb{H}$. These two trivializations induce equivariant  homeomorphisms 
$$
\varphi_{1}: S^{\mathbb{R}}\wedge  \operatorname{Th}(\Pin, N_{*}\Pin)\xrightarrow{\cong} \operatorname{Th}(\Pin, T_{*}\Pin\oplus N_{*}\Pin),
$$
$$
\varphi_{2}:  \operatorname{Th}(\Pin, T_{*}\Pin\oplus N_{*}\Pin)\xrightarrow{\cong} S^{\mathbb{H}}\wedge (\Pin_{+}).
$$
Let $\nu(\Pin)$ be a $\Pin$-equivariant, closed tubular neighborhood of $\Pin$ in $\mathbb{H}$. Then there exists an equivariant homeomorphism
$$
\varphi_{3}: \nu(\Pin)/\partial \nu(\Pin)\xrightarrow{\cong} \operatorname{Th}(\Pin, N_{*}\Pin).
$$
Then we consider the following composition
\begin{equation}\label{eq: space level description of alpha}
    \begin{split}
l: S^{\mathbb{R}}\wedge S^{\mathbb{H}}&\xrightarrow{\id_{S^{\mathbb{R}}}\wedge c_{1}} S^{\mathbb{R}}\wedge (\nu(\Pin)/\partial \nu(\Pin))\\
&\xrightarrow{\id_{S^{\mathbb{R}}}\wedge \varphi_{3}}S^{\mathbb{R}}\wedge \operatorname{Th}(\Pin, N_{*}\Pin)\\
&\xrightarrow{\varphi_{1}} \operatorname{Th}(\Pin, T_{*}\Pin\oplus N_{*}\Pin)\\
& \xrightarrow{\varphi_{2}} S^{\mathbb{H}}\wedge (\Pin_{+})\\
&\xrightarrow{\id_{S^{\mathbb{H}}}\wedge p_{3}} S^{\mathbb{H}}
\end{split}
\end{equation}
Here $c_{1}: S^{\mathbb{H}}\rightarrow \nu(\Pin)/\partial \nu(\Pin)$ is the Pontryagin--Thom collapsing map. The map $p_{3}: \Pin_{+}\rightarrow S^{0}$ preserves the base point and  collapses the whole $\Pin$ into the other point in $S^{0}$.
\begin{proposition}\label{prop: representative of alpha}
$\alpha_{0}$ is represented by the map $l$.
\end{proposition}
\begin{proof}
By the proof of Proposition \ref{prop: algebraic vanishing result2}, $\alpha_{0}=\zeta^{\Pin}_{\{e\}}(\alpha_{1})$, where  $$\zeta^{\Pin}_{\{e\}}:\pi_{1,\mathcal{U}}^{\Pin}(E\Pin_{+})\rightarrow \pi_{1,\mathcal{U}}^{\Pin}(S^{0})$$ 
is the map defined in (\ref{eq: zeta}) and $\alpha_{1}$ is any generator of $\pi_{1,\mathcal{U}}^{\Pin}(E\Pin_{+})$. We first give a space level description of $\alpha_{1}$. Consider the following commutative diagram  
\begin{equation}\label{diagram: EG homotopy group}
    \xymatrix{
\pi^{\Pin}_{1,\mathcal{U}}(\Pin_{+})  \ar[r]^{i_{1}} & \pi^{\Pin}_{1,\mathcal{U}}(E\Pin_{+})\\
\pi^{s}_{1}((S^{\widetilde{\mathbb{R}}}\wedge \Pin_{+})/\Pin)\ar[u]^{j}_{\cong} \ar[r]^{i'_{1}} & \pi^{s}_{1}((S^{\widetilde{\mathbb{R}}}\wedge E\Pin_{+})/\Pin)\ar[u]^{j'}_{\cong}}. 
\end{equation}
Here $i_{*}$ and $i'_{*}$ are both induced by the inclusion $i:\Pin_{+}\hookrightarrow E\Pin_{+}$. And $j$, $j'$ are the Adams' isomorphisms (\ref{eq: Adams' isomorphism}). By our discussion in the proof of Proposition \ref{prop: algebraic vanishing result2}, the group $\pi^{s}_{1}((S^{\widetilde{\mathbb{R}}}\wedge E\Pin_{+})/\Pin)$ is isomorphic to $\mathbb{Z}/2$ and it is generated by the inclusion of the $1$-skeleton: 
$$
S^{\mathbb{R}}\cong (S^{\widetilde{\mathbb{R}}}\wedge \Pin_{+})/\Pin\hookrightarrow (S^{\widetilde{\mathbb{R}}}\wedge E\Pin_{+})/\Pin.
$$
Therefore, the map $i'_{*}$ is just the surjective map $\mathbb{Z}\rightarrow \mathbb{Z}/2$. 

By the diagram (\ref{diagram: EG homotopy group}), we see that $\alpha_{1}=i_{*}\circ j(1)$. By the explicit description of the Adams' isomorphism \cite[Page 84]{Lewis-May-Steinberger1986}, we see that $j(1)\in \pi^{\Pin}_{1,\mathcal{U}}(\Pin_{+})$ is represented by the composition 
\begin{equation}\label{eq: l1}
l_{1}:=\varphi_{2}\circ \varphi_{1}\circ (\id_{S^{\mathbb{R}}}\wedge \varphi_{3})\circ (\id_{S^{\mathbb{R}}}\wedge c_{1}): S^{\mathbb{R}+\mathbb{H}}\rightarrow S^{\mathbb{H}}\wedge (\Pin_{+}). 
\end{equation}
(i.e., the composition of the first four maps in (\ref{eq: space level description of alpha})). Hence $\alpha_{1}$ is presented by the composition 
$$
l_{2}:S^{\mathbb{R}+\mathbb{H}}\xrightarrow{l_{1}}S^{\mathbb{H}}\wedge (\Pin_{+})\xrightarrow{\id_{S^{\mathbb{H}}}\wedge i}  S^{\mathbb{H}}\wedge (E\Pin_{+}).
$$
By the construction of $\zeta^{\Pin}_{\{e\}}$ in (\ref{eq: definition of zeta}), $\alpha_{0}=\zeta^{\Pin}_{\{e\}}(\alpha_{1})$ is represented by the composition 
$$
l_{3}: S^{\mathbb{R}+\mathbb{H}}\xrightarrow{l_{2}} S^{\mathbb{H}}\wedge (E\Pin_{+})\xrightarrow{\id_{S^{\mathbb{H}}} \wedge p_{0} } S^{\mathbb{H}}.
$$
 Since $p_{0}\circ i: \Pin_{+}\rightarrow S^{0}$ is exactly the map $p_{4}$, we see that the maps $l_{3}$ and $l$ are the same. \end{proof}
 
\begin{remark}
It is worth noting that $l$ actually represents an element in $\{S^{1},S^{0}\}^{\Pin}_{\mathcal{U'}}$, where $\mathcal{U'}=\mathbb{H}^{\infty}\oplus \widetilde{\mathbb{R}}^{\infty}$ is a universe that only contains the nontrivial representations $\widetilde{\mathbb{R}}$ and $\mathbb{H}$. Compared to the universe $\mathcal{U}= \mathbb{H}^{\infty}\oplus \widetilde{\mathbb{R}}^{\infty}\oplus \mathbb{R}^{\infty}$ (which we have been working with so far), this smaller universe actually fits better with our setting in gauge theory. This is because the trivializations (\ref{eq: trivialization}) only involve $\widetilde{\mathbb{R}}$ and $\mathbb{H}$. However, since $\mathcal{U}'$ doesn't contain the trivial representation $\mathbb{R}$, the corresponding stable homotopy category behaves badly. (For example, for general $X,Y$, the set $\{X,Y\}^{\Pin}_{\mathcal{U}'}$ does not even carry a group structure.) This is the main reason why we choose to work with $\mathcal{U}$ instead of $\mathcal{U}'$. We also note that when restricting to the subgroup $S^{1}$, the universes $\mathcal{U}$ and $\mathcal{U}'$ are isomorphic.
\end{remark}

We end the paper with the following proposition. It states that $\alpha_{0}$ cannot be detected by Borel (co)homology, which corresponds to various versions of monopole Floer (co)homology. 
\begin{proposition} Let $A$ be an abelian group and let $G$ be a closed subgroup of $\Pin$. Then the induced maps 
$$
\alpha_{0,*}: \tilde{H}^{G}_{*}(S^{\mathbb{R}+ \mathbb{H}};A)\rightarrow \tilde{H}^{G}_{*}(S^{\mathbb{H}};A),
$$
$$
\alpha^{*}_{0}: \tilde{H}_{G}^{*}(S^{\mathbb{H}};A)\rightarrow \tilde{H}_{G}^{*}(S^{\mathbb{R}+\mathbb{H}};A).
$$
between the $G$-equivariant reduced  Borel (co)homology groups (with coefficient $A$) are both trivial.
\end{proposition}
\begin{proof} $\alpha_{0}$ is represented by the map $l$. Since $l$ factors through the map 
$
l_{1}
$ in (\ref{eq: l1}). It suffices to show that the induced maps
$$
l_{1,*}: \tilde{H}^{G}_{*}(S^{\mathbb{R}+ \mathbb{H}};A)\rightarrow \tilde{H}^{G}_{*}(S^{\mathbb{H}}\wedge (\Pin_{+});A),
$$
$$
l^{*}_{1}: \tilde{H}_{G}^{*}((\Pin_{+})\wedge S^{\mathbb{H}};A)\rightarrow \tilde{H}_{G}^{*}(S^{\mathbb{R}+\mathbb{H}};A).
$$
are trivial. 
To prove this, we consider the fiber bundle $$
SU(2)/G\hookrightarrow BG\rightarrow BSU(2),
$$
which makes $H^{*}(BG;A)$ a module over $$H^{*}(BSU(2);\mathbb{Z})=H^{*}(\mathbb{HP}^{\infty};\mathbb{Z})=\mathbb{Z}[v].$$
An elementary argument using the Serre spectral sequence shows that the map 
$$
v\cdot -: H^{m}(BG;A)\rightarrow H^{m+4}(BG;A)
$$
is an isomorphism for any $m\geq 0$. By the definition of Borel cohomology, $\tilde{H}^{*}_{G}(S^{\mathbb{R}+\mathbb{H}};A)$ is module over $H^{*}(BG;\mathbb{Z})$. Hence $\tilde{H}_{G}^{*}(S^{\mathbb{R}+\mathbb{H}};A)$ is also over $\mathbb{Z}[v]$. By the Thom isomorphism theorem, we have an isomorphism 
$$
H^{*}(BG;A)\cong \tilde{H}_{G}^{*+5}(S^{\mathbb{R}+\mathbb{H}};A)
$$
between two $\mathbb{Z}[v]$-modules. Therefore, $\tilde{H}_{G}^{*}(S^{\mathbb{R}+\mathbb{H}};A)$ is supported in degrees greater or equal $5$ and  the map 
$$
v\cdot -: \tilde{H}_{G}^{m}(S^{\mathbb{R}+\mathbb{H}};A)\rightarrow \tilde{H}_{G}^{m+4}(S^{\mathbb{R}+\mathbb{H}};A)
$$
is an isomorphism for any $m\geq 5$. Moreover, since the $G$-action on $S^{\mathbb{H}}\wedge (\Pin_{+})$ is free away from the base point, the group 
$$\tilde{H}_{G}^{*}(S^{\mathbb{H}}\wedge (\Pin_{+});A)=\tilde{H}^{*}((S^{\mathbb{H}}\wedge (\Pin_{+})/G;A)$$
is supported in degrees less or equal to $5$ (less or equal to $4$ is $G$ is not a finite group). 

To this end, we see that $l^{*}_{1}$ can only be nontrivial in degree $5$. However, suppose $l^{*}_{1}(x)\neq 0$ for some $x\in \tilde{H}_{G}^{5}(S^{\mathbb{H}}\wedge (\Pin_{+});A)$. Then 
$$
0=l^{*}_{1}(0)=l^{*}_{1}(v\cdot x)=v\cdot l^{*}_{1}(x)\neq 0.
$$
This contradiction shows that $l^{*}_{1}=0$ is all degrees. Hence we have $\alpha^{*}_{0}=0$. The proof that $\alpha_{0,*}=0$ is essentially the same.
\end{proof}

\bibliography{references}
\bibliographystyle{plain}

\end{document}